\documentclass[reqno,12pt]{amsart}

\usepackage{amssymb,amsmath}
\usepackage{graphics}

\newif\ifnotesw\noteswtrue
\newcommand{\edit}[1]{\ifnotesw \marginpar%
 [{\scriptsize\it\begin{minipage}[t]{\marginparwidth}\raggedleft#1\end{minipage}}]%
 {\scriptsize\it\begin{minipage}[t]{\marginparwidth}\raggedright#1\end{minipage}} \fi}
\newtheorem{theorem}{Theorem}[section]
\newtheorem{lemma}[theorem]{Lemma}
\newtheorem{corollary}[theorem]{Corollary}
\newtheorem{example}[theorem]{Example}
\newtheorem{definition}[theorem]{Definition}
\newenvironment{bfenumerate}%
{\begin{enumerate}}%
{\end{enumerate}}

\newenvironment{proofof}[1]{\par\smallbreak\noindent{\it Proof~of~#1.}}%
{\unskip\nobreak\hfill \qed \par\medbreak}
\newenvironment{sketch}{\par\smallbreak\noindent{\it Sketch of Proof.~}}
{\unskip\nobreak\hfill \qed \par\medbreak}
\newcounter{claim}
\renewcommand{\theclaim}{\Alph{claim}}
\newenvironment{claim}{\refstepcounter{claim}%
\par\medskip\par\noindent{\it Claim~\theclaim.~}~\rm}%
{\par\smallskip\par}
\newenvironment{subproof}{\par\smallskip\par\noindent{\sl Proof of Claim~\theclaim.~}}%
{$\,\triangleleft$}

\renewcommand{\paragraph}[1]{\smallskip\noindent{\it #1}~}

\newcommand{\refeq}[1]{(\ref{#1})}
\newcommand{\setdef}[2]{\left\{ \hspace{0.5mm} #1 : \hspace{0.5mm} #2 \right\}}
\newcommand{\msetdef}[2]{\left\{\!\!\left\{ \hspace{0.5mm} #1 : \hspace{0.5mm} #2 \right\}\!\!\right\}}
\newcommand{\function}[2]{:#1 \rightarrow #2}
\newcommand{\of}[1]{\left( #1 \right)}
\newcommand{\ofc}[1]{\left\{ #1 \right\}}
\newcommand{\feq}{\stackrel{\mbox{\tiny def}}{=}}
\newcommand{\und}{\wedge}
\newcommand{\Or}{\vee}
\newcommand{\A}{\forall}
\newcommand{\E}{\exists}
\newcommand{\compl}[1]{\overline{#1}}
\newcommand{\cd}[2]{D_{\#}^{#1}(#2)}
\newcommand{\cw}[1]{W_{\#}(#1)}
\newcommand{\game}{\mbox{\sc Ehr}}
\newcommand{\ga}{\alpha}
\newcommand{\gb}{\beta}
\newcommand{\keq}{\mathbin{{\equiv}_k}}
\newcommand{\oneeq}{\mathbin{{\equiv}_1}}
\newcommand{\kkeq}{\mathbin{{\equiv}_{k,k}}}
\newcommand{\kseq}{\mathbin{{\equiv}_{k,s}}}
\newcommand{\notkseq}{\mathbin{{\not\equiv}_{k,s}}}
\newcommand{\sseq}{\mathbin{{\equiv}_{k,0}}}
\newcommand{\ksseq}{\mathbin{{\equiv}_{k,s+1}}}

\newcommand{\baru}{{\bar u}}
\newcommand{\barv}{{\bar v}}
\newcommand{\isotype}[1]{{\textup{tp}(#1)}}
\newcommand{\wl}[3]{C^{#1,#2}(#3)}
\newcommand{\diagwl}[3]{\mathit{diag\hspace{0.6pt}C}^{#1,#2}(#3)}
\newcommand{\wwl}[2]{C^{#1,#2}}
\newcommand{\stabi}[2]{\mathit{Stab}^{#1}(#2)}
\newcommand{\WL}[1]{$#1$-dim WL}
\newcommand{\tlt}[1]{\par{\it #1}}
\newcommand{\mso}{\{\!\!\{}
\newcommand{\msc}{\}\!\!\}}
\newcommand{\tc}[1]{\mbox{\rm TC$^{#1}$}}
\newcommand{\nc}[1]{\mbox{\rm NC$^{#1}$}}
\newcommand{\ac}[1]{\mbox{\rm AC$^{#1}$}}
\newcommand{\xeq}{\mathbin{{\equiv}_X}}
\newcommand{\calC}{\mathcal{C}}
\newcommand{\sieve}{\mathcal{S}}

\newcommand{\dist}{\mathit{dist}}
\newcommand{\diam}{\mathit{diam}}
\newcommand{\rgraph}{G_{n,\scriptscriptstyle 1/2}}
\newcommand{\rpgraph}{G_{n,p}}
\newcommand{\e}{\mathrm{e}}
\newcommand{\prob}[1]{\mathbb{P}\left[\, #1\, \right]}
\newcommand{\expect}[1]{\mathbb{E}\left[\, #1\, \right]}
\newcommand{\condexpect}[2]{\mathbb{E}\left[ #1 \mid #2 \right]}
\newcommand{\tower}{\mathit{Tower}}
\newcommand{\prenex}{{\mbox{\scriptsize\it prenex}}}
\newcommand{\sshift}{\mbox{\qquad\qquad\qquad\qquad\quad}}
\newcommand{\bs}[1]{\mathit{BS}(#1)}
\newcommand{\bflex}{\sl}

\title{Logical complexity of graphs: a survey}

\author{Oleg Pikhurko\,$^*$}\thanks{$^*$\,%
Department of Mathematical Sciences,
Carnegie Mellon University, Pittsburgh, PA 15213, USA.
This work done under the support of the
National Science Foundation (Grant DMS-0758057) and the Alexander von Humboldt Foundation.}

\author{Oleg Verbitsky\,$^\dag$}\thanks{$^\dag$\,%
Institute for Applied Problems of Mechanics and Mathematics, 
79060 Lviv, Ukraine.
This work was done under the support of the Alexander von Humboldt Foundation.}

\date{}
\begin{document}

\sloppy

\begin{abstract}
We discuss the definability of finite graphs in
first-order logic with two relation symbols
for adjacency and equality of vertices.
The \emph{logical depth} $D(G)$ of a graph $G$ is equal to
the minimum quantifier depth of a sentence defining $G$
up to isomorphism. The \emph{logical width} $W(G)$
is the minimum number of variables occurring in
such a sentence. The \emph{logical length} $L(G)$
is the length of a shortest defining sentence.
We survey known estimates for these
graph parameters and discuss their relations
to other topics (such as the efficiency of the
Weisfeiler-Lehman algorithm in isomorphism testing,
the evolution of a random graph, 
quantitative characteristics of the zero-one law,
or the contribution of 
Frank Ramsey to the research on Hilbert's
\emph{Entscheidungsproblem}).
Also, we trace 
the behavior of the descriptive complexity
of a graph as the logic becomes more restrictive
(for example, only definitions with a bounded number of variables or 
quantifier alternations are allowed) or more expressible
(after powering with counting quantifiers).
\end{abstract}

\maketitle
\markleft{\sc OLEG PIKHURKO and OLEG VERBITSKY}

\break

\tableofcontents

\break

\section{Introduction}

\subsection{Basic notions and examples}

We consider the first-order language of graph theory whose vocabulary contains
two relation symbols $\sim$ and $=$, respectively for adjacency and equality of vertices.
The term \emph{first-order} imposes the condition that
the variables represent vertices and hence the quantifiers apply to vertices only.
Without quantification over sets of vertices, we are unable to express
by a single formula
some basic properties of graphs, such as being bipartite, being connected, etc.
(see, e.g.,~\cite[Theorems~2.4.1 and~2.4.2]{Spe}).
However, first-order logic is powerful enough to define any \emph{individual}
graph. How succinctly this can be done is the subject of this article.

As a starting example,
let us say in the first-order language that vertices $x$ and $y$ are at
distance at most $n$ from one another. A possible formula $\Delta_n(x,y)$ can look
as follows:
\begin{eqnarray}
\Delta_1(x,y)&\feq&x\sim y\Or x=y,\nonumber\\
\Delta_n(x,y)&\feq&\exists z_1\ldots\exists z_{n-1}
\Bigl(
\Delta_1(x,z_1)\und\bigwedge_{i=1}^{n-2}\Delta_1(z_i,z_{i+1})\und
\Delta_1(z_{n-1},y)
\Bigr).\label{ex:dist_n}
\end{eqnarray}

By a \emph{sentence} we mean a first-order formula where every variable is 
bound by a quantifier.
If we specify a graph $G$, a sentence $\Phi$ is either \emph{true} or 
\emph{false} on it. If $H$ is a graph isomorphic to $G$, then $\Phi$
is either true or false on $G$ and $H$ simultaneously. In other words,
first-order logic cannot distinguish between isomorphic graphs.
In general, we say that a sentence $\Phi$ \emph{distinguishes} a graph $G$
from another graph $H$ if $\Phi$ is true on $G$ but false on~$H$.

For example, 
sentence $\A x\A y\,\Delta_1(x,y)$ distinguishes a complete graph $K_n$
from any graph $H$ that is not complete.
The sentence $\A x\A y\,\Delta_{n-1}(x,y)$ distinguishes $P_n$, the path with $n$ vertices,
from any longer path $P_m$, $m>n$.

Throughout this survey we consider only graphs whose vertex set is
finite and non-empty. 
We say that a sentence $\Phi$ {\em defines} a graph $G$ (up to isomorphism) 
if $\Phi$ distinguishes $G$ from every non-isomorphic graph $H$.

For example, the single-vertex graph $P_1$ is defined 
by sentence $\A x\A y\, (x=y)$.
If $n\ge2$, then the path $P_n$ is defined by
\begin{equation}\label{ex:path}
\mbox{}\hspace{-5mm}%
\begin{array}{rcl}
&&\forall x\forall y \Delta_{n-1}(x,y)\und\neg\forall x\forall y \Delta_{n-2}(x,y)\\[2.5mm]
&&\sshift\mbox{\tt to say that the diameter equals $n-1$}\\[1.5mm]
&&\displaystyle{}\und\forall x\neg\exists y_1\exists y_2\exists y_3
\of{\bigwedge_{i=1,2,3}x\sim y_i\und
\bigwedge_{i\ne j}\neg(y_i=y_j)}\\[5.5mm]
&&\sshift\mbox{\tt to say that the maximum degree $\le2$}\\[2mm]
&&\displaystyle{}\und\exists x\neg\exists y_1\exists y_2
\of{\bigwedge_{i=1,2}x\sim y_i\und\neg(y_1=y_2)}\\[3mm]
&&\sshift\mbox{\tt to say that the minimum degree $\le1$ (thereby}\\
&&\sshift\mbox{\tt distinguishing from cycles $C_{2n-2}$ and $C_{2n-1}$)}
\end{array}
\end{equation}

We have already mentioned the following basic fact: {\it Every finite graph $G$ is definable.}%
\footnote{This fact, though very simple, highlights a fundamental difference
between the finite and the infinite: There are non-isomorphic countable
graphs satisfying precisely the same first-order sentences
(see, e.g., \cite[Theorem 3.3.2]{Spe}).}
Indeed, let $V(G)=\{v_1,\dots,v_n\}$ be the vertex set of $G$ and $E(G)$
be its edge set. A sentence defining $G$ could read:
 \renewcommand{\arraystretch}{1.4}
 \begin{equation}\label{eq:def}
 \begin{array}{l}
\exists x_1\dots \exists x_n\ \left(\,\mathrm{Distinct}(x_1,\dots,x_n)\wedge
 \mathrm{Adj}(x_1,\dots,x_n)\,\right)\\
\qquad\wedge\ \forall x_1\dots\forall x_{n+1}\ \neg\,
 \mathrm{Distinct}(x_1,\dots,x_{n+1}),\end{array}
 \end{equation}
 where, for the notational convenience, we use the following shorthands
 \begin{eqnarray*}
\mathrm{Distinct}(x_1,\dots,x_k)&\feq&\bigwedge_{1\le i<j\le k} \neg\, (x_i=x_j),\\[2mm]
 \mathrm{Adj}(x_1,\dots,x_n)&\feq&\bigwedge_{\{v_i,v_j\}\in E(G)} x_i\sim
 x_j\ \wedge\bigwedge_{\{v_i,v_j\}\not\in E(G)} \neg\, (x_i\sim  x_j).
 \end{eqnarray*}
 In other words, we first specify that there are $n$ distinct vertices,
list the adjacencies and the non-adjacencies between them, and then state that 
we cannot find $n+1$ distinct vertices.

The sentence \refeq{eq:def} is an exhaustive description of $G$ and seems rather
wasteful. We want to know if there is a more succinct way of defining a graph
on $n$ vertices. The following natural succinctness measures of a first-order formula $\Phi$
are of interest: 
 \begin{itemize}
 \item the {\em length\/} $L(\Phi)$ which is the total number of
symbols in $\Phi$ (each variable symbol contributes 1);
 \item the {\em quantifier depth\/} $D(\Phi)$ which is the maximum
length of a chain of nested quantifiers in $\Phi$;
 \item the {\em width\/} $W(\Phi)$ which is the 
number of variables used in $\Phi$ (different occurrences of the same variable 
are not counted).\footnote{%
Gr\"adel \cite{Gra} defines the width of a formula $\Phi$
as the maximum number of free variables in a subformula of $\Phi$.
Denote this version by $W'(\Phi)$. Clearly, $W'(\Phi)\le W(\Phi)$
and the inequality can be strict. Nevertheless, the two parameters
are closely related: $\Phi$ can be rewritten by renaming bound variables
in an equivalent form $\Phi'$ so that $W(\Phi')=W'(\Phi)$; see \cite[Lemma 3.1.4]{Gra}.}
 \end{itemize}

Formula $\Delta_n$ in \refeq{ex:dist_n} was intentionally written
in a non-optimal way. Note that
$L(\Delta_n)=\Theta(n)$, $D(\Delta_n)=n-1$, and $W(\Delta_n)=n+1$.
The same distance restriction can be expressed more succinctly with respect to
the latter two parameters, namely
\begin{equation}\label{ex:deltaa}
\begin{array}{rcl}
\Delta'_1(x,y)&\feq&\Delta_1(x,y),\\
\Delta'_n(x,y)&\feq&\exists z\of{\Delta'_{\lfloor n/2\rfloor}(x,z)\und
\Delta'_{\lceil n/2\rceil}(z,y)},
\end{array}
\end{equation}
 where $\lceil x\rceil$ (resp.\ $\lfloor x\rfloor$) 
stands for the integer nearest
to $x$ from above (resp.\ from below).
Now $D(\Delta'_n)=\lceil\log_2n\rceil$, giving an exponential gain for the quantifier depth!
The width can be reduced even more drastically: by recycling variables
we can write $\Delta'_n$ with only $3$ variables in total, achieving 
$W(\Delta'_n)=3$.

We now come to the central concepts of our survey.
Let us define $L(G)$ (resp.\ $D(G)$, $W(G)$) to be the minimum of $L(\Phi)$ 
(resp.\ $D(\Phi)$, $W(\Phi)$) over all sentences $\Phi$ defining a graph $G$.
We will call these graph invariants, respectively,
the \emph{logical length, depth}, and \emph{width} of~$G$.

\begin{example}\label{ex:LDW}\mbox{}\rm

\begin{bfenumerate}
\item
Using $\Delta'_n$ in place of $\Delta_n$ in \refeq{ex:path},
we see that $D(P_n)<\log_2n+3$ and $W(P_n)\le 4$. The reader
is encouraged to improve the latter to $W(P_n)\le 3$.
\item
The generic defining sentence \refeq{eq:def} shows that $L(G)=O(n^2)$ and $D(G)\le n+1$
for every graph $G$ on $n$ vertices.
\item
The {\em complement\/} of $G$, denoted by $\compl G$, is
the graph on the same vertex set $V(G)$ whose edges are those pairs
that are
not in $E(G)$. One can easily prove that $D(\compl G)=D(G)$ and $W(\compl G)=W(G)$.
\end{bfenumerate}
\end{example}

The logical length, depth, and width of a graph satisfy the following
inequalities: 
$$
W(G)\le D(G) < L(G).
$$
The latter relation follows from an obvious fact that $D(\Phi)<L(\Phi)$ for
any first-order formula $\Phi$. The former follows from a bit less obvious fact that
for any first-order formula $\Phi$ there is a logically equivalent formula
$\Psi$ with $W(\Psi)\le D(\Phi)$.

\subsection{Variations of logic}
\subsubsection{Fragments}

Suppose that we put some restrictions on the structure of a defining sentence.
This may cause an increase in the resources (length, depth, width)
that we need
in order to define a graph in the straitened circumstances. These 
effects will be one of our main
concerns in this survey. We will deal with restrictions of the
following
two sorts. We may be allowed to make only a small (constant) number of quantifier alternations or
to use only a bounded number of variables. The former is commonly
used in logic and complexity theory to obtain hierarchical classifications
of various problems. The latter is in the focus of
\emph{finite-variable logics} (see, e.g, Grohe~\cite{Gro:survey}).
Moreover, the number of variables has relevance to the computational
complexity of the graph isomorphism problem, see Section~\ref{s:wl}.

\paragraph{Bounded number of quantifier alternations.}
A first-order formula $\Phi$ with connectives $\{\neg,\wedge,\vee\}$ is in a
\emph{negation normal form} if all negations apply only to relations
(one can think that we now do not have negation at all but  introduce  instead
two new relation symbols, for inequality and non-adjacency).
It is well known that this structural restriction actually does not make
first-order logic weaker: We can always move negations in front of
relation symbols without increasing the formula's length more than twice
and without changing the quantifier depth and the width.

Given such a formula $\Phi$ and a sequence of nested quantifiers
in it, we count the number of \emph{quantifier alternations},
that is, the number of successive pairs $\A\E$ and $\E\A$ in the sequence.
The \emph{alternation number} of $\Phi$ is the maximum number
of quantifier alternations over all such sequences.
The \emph{$a$-alternation logic} consists of all first-order
formulas in the negation normal form whose alternation number does not exceed $a$.
We will adhere to the following notational convention:
a subscript $a$ will always indicate that at most $a$ quantifier
alternations are allowed. For example, $D_a(G)$ is the minimum quantifier depth
of a sentence in the $a$-alternation logic that defines a graph~$G$.

For any graph $G$ on $n$ vertices we have
$$
D(G)\le\ldots\le D_{a+1}(G)\le D_a(G)\le\ldots\le D_1(G)\le D_0(G)\le n+1,
$$
where the last bound is due to the defining sentence~\refeq{eq:def}.

\paragraph{Bounded number of variables.}
The \emph{$k$-variable logic} is the fragment of first-order logic where
only $k$ variable symbols are available, that is, the formula width
is bounded by $k$. 
The restriction of defining sentences to the $k$-variable logic
will be always indicated by a superscript $k$.
To make this notation always applicable, we set
$D^k(G)=\infty$ if the $k$-variable logic is too
weak to define $G$.
If $k\ge W(G)$ for a graph $G$ of order $n$, then we have
$$
D(G)\le D^{k+1}(G)\le D^k(G) < n^{k-1}+k,
$$
where the last bound will be established in Theorem \ref{thm:Dk} below.
Note that the bounds in Example \ref{ex:LDW}.1 can be strengthened to $D^3(P_n)<\log_2n+3$.

\subsubsection{An extension with counting quantifiers}

We will also enrich first-order logic by allowing one to use expressions of the type
$\exists^m\Psi$ in order to say that there are at least $m$ vertices with property 
$\Psi$. Those are called \emph{counting quantifiers} and the extended logic 
will be referred to as \emph{counting logic}.
A counting quantifier $\exists^m$ contributes 1 in the quantifier depth irrespectively 
of the value of $m$. For the counting logic we will use the ``sharp-notation'',
thus denoting the logical depth and width of a graph $G$ in this logic, 
respectively, by $\cd{}G$ and $\cw G$. Clearly,
$\cd{}G\le D(G)$ and $\cw G\le W(G)$. The counting quantifiers often allow
us to define a graph much more succinctly. For example, $\cd{}{K_n}=\cw{K_n}=2$ as
this graph is defined by 
$$
\A x\A y\, (x\sim y\vee x=y)\und \E^n x\,(x=x)\und\neg \E^{n+1} x\,(x=x).
$$
This is in sharp contrast with the fact that $D(K_n)=W(K_n)=n+1$,
where the lower bound follows from 
the simple observation that $n$ variables are not enough to
distinguish between $K_n$ and~$K_{n+1}$.

\subsection{Outline of the survey}

Section \ref{s:prelim} specifies notation
and proves a couple of basic facts about first-order
sentences. The latter are applied to
establish an upper bound on the logical length $L(G)$
of a graph in terms of its logical depth $D(G)$
and to estimate from above the number of graphs
whose logical depth is bounded by a given parameter~$k$.
The existence of such bounds is more important than
the bounds themselves that are huge, involving
the tower function. Furthermore, we define $D(G,H)$ to
be the smallest quantifier depth sufficient to distinguish
between non-isomorphic graphs $G$ and $H$. We will observe that
the obvious inequality $D(G,H)\le D(G)$ 
gives the sharp lower bound on
$D(G)$. Thus estimating $D(G)$ reduces to estimating~$D(G,H)$ for all
$H\not\cong G$

The value of $D(G,H)$ is characterized in Section \ref{s:games}
as the length of the \emph{Ehrenfeucht game} on $G$ and $H$.
Moreover, the logical width admits a characterization in terms of another
parameter of the game. Thus, the determination of the logical depth
and width of a graph reduces to designing optimal strategies
in the Ehrenfeucht game.

In Section \ref{s:wl}, the logical width and the logical depth 
are also characterized, respectively,
as the minimum \emph{dimension} and the minimum \emph{number of rounds}
such that the so-called \emph{Weisfeiler-Lehman algorithm} returns the
correct answer. The algorithm tries
to decide whether two input graphs are isomorphic; its one-dimensional
version is just the well-known color-refining procedure.
Thus, an analysis of the algorithm  can give us information
on the logical complexity of the input graphs. This
relationship is even more advantageous  in the other direction:
Once we prove that all graphs in some class $C$ have low logical
complexity, we immediately obtain an efficient isomorphism test for~$C$.

This paradigm is successful for
graphs with bounded treewidth and planar graphs, with good prospects
for covering all classes of graphs with an excluded minor.
In Section \ref{ss:classes} we report strong upper bounds
for the logical depth/width of graphs in these classes.
In Section \ref{s:general} we survey the bounds known in the general case.
In particular, if a graph $G$ on $n$ vertices has no \emph{twins}, 
i.e., no two vertices have the same adjacency to the rest of the graph, 
then $D(G)<\frac12n+3$. The factor of $\frac12$ can be improved for
graphs with bounded vertex degrees. Here we have to content ourselves
with linear bounds in view of a linear lower bound by Cai, F\"urer,
and Immerman \cite{CFI}. They constructed examples of graphs with
maximum degree~3 such that $\cw G > c\,n$ for a positive constant~$c$.

Section \ref{s:average} discusses the logical complexity of a random
graph. We obtain rather close lower and upper bounds for \emph{almost all}
graphs. Furthermore, we trace the behavior of the logical depth
in the evolutional random graph model~$\rpgraph$ 
where $p$ is a function of $n$. 

While in Sections \ref{s:worstcase} and \ref{s:average} we deal with,
respectively, worst case and average case bounds, Section \ref{s:best}
is devoted to the \emph{best case}.  More specifically, we define
\emph{succinctness function} $q(n)$ to be equal to the minimum of $D(G)$
over all $G$ on $n$ vertices.  Since only finitely many graphs are
definable with a fixed quantifier depth, $q(n)$ goes to infinity
as $n$ increases.  It turns out that its growth is inconceivably slow:
We show a superrecursive gap between the values of $q(n)$ and $n$.
This phenomenon disappears if we ``smoothen'' $q(n)$ by considering
the least monotonic upper bound for this function: 
the smoothed succinctness function is very close to the log-star
function. Furthermore, the succinctness function can be considered in
any logic. Let $q_0(n)$ be its variant for the logic with no
quantifier alternation. We can determine $q_0(n)$ with rather high
precision: It is also related to the log-star function.  The
lower bound for $q_0(n)$ implies
a superrecursive gap between the graph parameters $D(G)$ and $D_0(G)$,
yet another evidence of the weakness of the 0-alternation logic.
The tight upper bound for $q_0(n)$ shows that, nevertheless,
there are graphs whose definitions, even if
quantifiers are not allowed to alternate, can have surprisingly low
quantifier depth. We give several methods of explicit constructions of
such graphs. These constructions have another interesting aspect.
They allow us to show that the previously mentioned
tower-function bounds from Section \ref{s:prelim} cannot be improved
substantially.

Some of the most interesting open questions are collected in
Section~\ref{s:open}.

\subsection{Other structures}

Some of the results presented in the survey generalize 
to relational structures over a fixed vocabulary.
Such generalizations are often straightforward. For example,
the upper bounds on succinctness functions hold true if
the vocabulary contains at least one
relation symbol of arity more than 1 (since any graph can be trivially
represented as a structure over this vocabulary).
Extension of the worst case bounds to general structures is also possible
but requires essential additional efforts; see~\cite{PVe}.

Various definability parameters were investigated also for
special structures:
colored graphs (Immerman and Lander \cite{ILa}, 
Cai, F\"urer, and Immerman \cite{CFI}),
digraphs and hypergraphs (Pikhurko, Veith, and Verbitsky \cite{PVVarxiv}),
bit strings and ordered trees (Spencer and St.~John \cite{SJo}),
linear orders (Grohe and Schweikardt~\cite{GSc}).

\section{Preliminaries}\label{s:prelim}

\subsection{Notation: Arithmetic and graphs}

We define the {\em tower function\/} by $\tower(0)=1$ and
$\tower(i)=2^{\tower(i-1)}$ for each subsequent integer $i$.
Given a function $f$, by $f^{(i)}(x)$ 
we will denote the $i$-fold composition of $f$. 
In particular, $f^{(0)}(x)=x$.
By $\log n$ we always mean the logarithm base 2. The ``inverse'' of the tower
function, the {\em log-star\/} function $\log^*n$, is defined
by $\log^*n=\min\setdef{i}{\tower(i)\ge n}$.
We use the standard asymptotic notation. For example, 
$f(n)=\Omega(g(n))$ means that there is a constant $c>0$
such that $f(n)\ge c\,g(n)$ for all sufficiently large~$n$.

The number of vertices in a graph $G$ is called the \emph{order}
of $G$ and is denoted by $v(G)$. The \emph{neighborhood} $N(v)$ of a vertex $v$
consists of all vertices adjacent to $v$. The \emph{degree} of $v$
is defined by $\deg v=|N(v)|$. The \emph{maximum degree} of a graph $G$
is defined by $\Delta(G)=\max_{v\in V(G)}\deg v$.

The \emph{distance} between vertices $u$ and $v$ in a graph $G$ is 
defined to be the minimum length of a path from $u$ to $v$ and denoted
by $\dist(u,v)$. 
If $u$ and $v$ are in different connectivity components, then 
we set $\dist(u,v)=\infty$.
The {\em eccentricity\/} of a vertex $v$ is defined
by $e(v)=\max_{u\in V(G)}\dist(v,u)$. 

Let $X\subset V(G)$. The subgraph induced by $G$ on $X$ is denoted by $G[X]$.
We denote $G\setminus X=G[V(G)\setminus X]$, which is the result of
the removal of all vertices in $X$ from $G$. 
If a single vertex $v$ is removed, we write $G-v=G\setminus\{v\}$.
A set of vertices $X$ is called \emph{homogeneous} if $G[X]$
is a complete or an empty graph.

A graph is \emph{$k$-connected} if it has at least $k+1$ vertices and 
remains connected after removal of any $k-1$ vertices.
2-connected graphs are also called \emph{biconnected}.

A graph is \emph{asymmetric} if it admits no non-trivial automorphism.

\subsection{A length-depth relation}

We have already mentioned the trivial relation $D(G)<L(G)$.
Now we aim at bounding $L(G)$ from above in terms of $D(G)$.
We write $G\keq H$ to say that graphs $G$ and $H$
cannot be distinguished by any sentence with quantifier depth $k$.
As it is easy to see, $\keq$ is an equivalence relation.
Its equivalence classes will be referred to as \emph{$\keq$-classes}.
We say that a sentence $\Phi$ \emph{defines} a $\keq$-class $\alpha$
if $\Phi$ is true on all graphs in $\alpha$ and false on all other graphs.

\begin{lemma}\label{lem:keq}\mbox{}
\begin{bfenumerate}
\item
The number of $\keq$-classes is finite and does not exceed
$\tower(k+\log^*k+2)$. 
\item
Every $\keq$-class is definable by a sentence $\Phi$ with
$D(\Phi)=k$ and $L(\Phi)<\tower(k+\log^*k+2)$.
\end{bfenumerate}
\end{lemma}

\begin{proof} 
The case of $k=1$ is easy: There is only one $\oneeq$-class
(consisting of all graphs), which is
definable by $\A x(x=x)$.

Let $k\ge2$ and $0\le s\le k$.
When we write $\bar z$, we will mean an $s$-tuple $(z_1,\ldots,z_s)$
(if $s=0$, the sequence is empty).
If $\bar u\in V(G)^s$ and $\Phi$ is a formula with $s$ free variables
$x_1,\dots,x_s$, then 
notation $G,\bar u\models\Phi(\bar x)$ will mean that
$\Phi(\bar x)$ is true on $G$ with each $x_i$ being assigned the
respective~$u_i$ as its value.

A formula $\Phi(x_1,\ldots,x_s)$ of quantifier depth $k-s$ is {\em normal\/} if
$\Phi$ is built from variables $x_1,\ldots,x_k$ and
every maximal sequence of nested quantifiers in $\Phi$ has length $k-s$ and
quantifies the variables $x_{s+1},\ldots,x_k$ exactly in this order.
A simple inductive syntactic argument shows that any
$\Phi(x_1,\ldots,x_s)$ has an equivalent normal formula $\Phi'(x_1,\ldots,x_s)$
of the same quantifier depth as $\Phi$.

We write $G,\bar u\kseq H,\bar v$ to say that
$G,\bar u\models\Phi(\bar x)$ exactly when $H,\bar v\models\Phi(\bar x)$
for every normal formula $\Phi$ of quantifier depth $k-s$.
A normal formula $\Phi(\bar x)$ \emph{defines} a $\kseq$-class $\alpha$ if
$G,\bar u\models\Phi(\bar x)$ exactly when $G,\bar u$ belongs to $\alpha$.
The $\kseq$-equivalence class of $G,\bar u$ will be denoted by
$[G,\bar u]_{k,s}$. 

Let $f(k,s)$ denote the number of all $\kseq$-classes and $l(k,s)$ denote
the minimum $l$ such that every $\kseq$-class is definable by
a normal formula of depth at most $k-s$ and length at most $l$. 
Note that relations $\keq$ and $\sseq$ coincide.
Thus, our goal is to estimate the numbers $f(k,0)$ and $l(k,0)$ from above.

We use the backward induction on $s$.
A $\kkeq$-class can be determined by specifying, for each pair of the $k$
elements, whether they are equal and, if not, whether they are
adjacent or non-adjacent. There are at most three choices per pair.  It easily follows that
$f(k,k)\le 3^{{k\choose 2}}$ and $l(k,k)<9k^2$. 
We are now going to estimate $f(k,s)$ and $l(k,s)$
in terms of $f(k,s+1)$ and $l(k,s+1)$. Suppose that each $\ksseq$-class $\beta$
is defined by a formula $\Phi_\beta(x_1,\ldots,x_s,x_{s+1})$ whose length
is bounded by~$l(k,s+1)$.

Define $S(G,\bar u)=\setdef{[G,\bar u,u]_{k,s+1}}{u\in V(G)}$, the set of $\ksseq$-classes
obtainable from $G,\bar u$ by specifying one extra vertex. Note that
$$
G,\bar u\kseq H,\bar v\mbox{\ \ if and only if\ \ }S(G,\bar u)=S(H,\bar v).
$$
Indeed, suppose that $S(G,\bar u)\ne S(H,\bar v)$, say, 
$\beta=[G,\bar u,u]_{k,s+1}$ is not in $S(H,\bar v)$ for some $u\in V(G)$.
Then $G,\bar u\notkseq H,\bar v$ because formula $\E x_{s+1}\Phi_\beta$
is true for $G,\bar u$ but false for $H,\bar v$.
Suppose now that $G,\bar u$ and $H,\bar v$ are distinguishable
by a normal formula of quantifier depth $k-s$.
As it is easily seen, they are distinguishable by such a formula
of the form $\E x_{s+1}\Phi$. Without loss of generality, assume that
the formula $\E x_{s+1}\Phi$ is true for $G,\bar u$ but false for $H,\bar v$.
Let $u\in V(G)$ be such that $G,\bar u,u\models\Phi$. Since $\Phi$
distinguishes $G,\bar u,u$ from all $H,\bar v,v$ with $v\in V(H)$, the class
$[G,\bar u,u]_{k,s+1}$ is not in $S(H,\bar v)$ and, hence, $S(G,\bar u)\ne S(H,\bar v)$.

Thus, for a $\kseq$-class $\alpha$ we can correctly define the set of
$\ksseq$-classes accessible from $\alpha$ by $S(\alpha)=S(G,\bar u)$
for some (in fact, arbitrary) $G,\bar u$ in $\alpha$. It follows from
what we have proved that for arbitrary $\kseq$-classes $\alpha$ and
$\alpha'$, we have
$$
\alpha=\alpha'\mbox{\ \ if and only if\ \ }S(\alpha)=S(\alpha').
$$
As an immediate consequence,
$$
f(k,s)\le 2^{f(k,s+1)}.
$$
 Since $2\cdot {k\choose 2}\le 2^k$ for every integer $k\ge1$, we have
$f(k,k)\le 2^{2^k}\le \tower(\log^*k+2)$. By the above recursion, we
conclude that $f(k,0)\le\tower(k+\log^*k+2)$, which proves Part~1 of the lemma.

Another conclusion is that any $\kseq$-class $\alpha$ can be defined
by a normal formula\footnote{This is a variant of Hintikka's formula, cf.~\cite[Definition 2.2.5]{EFl}.}
$$
\Phi_\alpha(\bar x)\feq
\bigwedge_{\beta\in S(\alpha)}
\exists x_{s+1}\, \Phi_\beta(\bar x,x_{s+1})
\ \wedge\ \forall x_{s+1}
\bigwedge_{\beta\not\in S(\alpha)} \neg\, \Phi_\beta(\bar x,x_{s+1}).
$$
Looking at the length of $\Phi_\alpha(\bar x)$, we obtain the recurrence
 \begin{equation}\label{rec_l}
 l(k,s)\le f(k,s+1)(l(k,s+1)+9).
 \end{equation}
 Set $g(x)=2^x(x+9)$. A simple inductive argument shows that 
 $$
 f(k,s)\le 2^{g^{(k-s)}(9k^2)}\quad \mbox{and}\quad l(k,s)\le
g^{(k-s)}(9k^2).
 $$
 Define the two-parameter function 
$\tower(i,x)$ inductively on $i$ 
by $\tower(0,x)=x$ and $\tower(i+1,x)=2^{\tower(i,x)}$ for $i\ge 0$.
This is a generalization of the old function:
$\tower(i,1)=\tower(i)$. One can prove by induction on $i$ that for any
$x\ge 5$ and $i\ge 1$ we have 
 \begin{equation}\label{g}
 g^{(i)}(x)< \tower(i+1,x)/2.
 \end{equation} 
 Indeed, it is easy to check the validity of \refeq{g} for $i=1$, while for
$i\ge 2$ we have
 \begin{equation}\label{indg}
 g^{(i)}(x)< g(\tower(i,x)/2) < 2^{\tower(i,x)-1} = \tower(i+1,x)/2.
 \end{equation}
We have for all $k\ge 5$ that $9k^2< \tower(\log^*k+1)$. This follows
from $9k^2<2^k$ for $k\ge 10$ and can be checked by hand for $5\le k\le
9$.
 Thus, for $k\ge 5$, we have by~\refeq{g} that
 $$
 l(k,0)\le g^{(k)}(9k^2) < \tower(k+1,9k^2)/2 < \tower(k+1,9k^2)
< \tower(k+\log^*k+2).
 $$
 Routine calculations (omitted) based on~\refeq{rec_l} and the
exact initial values $f(2,2)=3$, $f(3,3)=15$, and $f(4,4)=127$ give Part~2 of
the lemma for $2\le k\le 4$.
\end{proof}

Lemma \ref{lem:keq}.2 gives us a bound for the logical length
of a graph in terms of its logical depth. It suffices to notice that each single
graph $G$ constitutes a $\keq$-class for $k=D(G)$.

\begin{theorem}[Pikhurko, Spencer, and Verbitsky~\cite{PSV}]\label{thm:dvsl}
 $$
 L(G)<\tower(D(G)+\log^* D(G)+2).$$
\end{theorem}

In fact, \cite[Theorem~10.1]{PSV} states only that
$L(G)<\tower(D(G)+\log^* D(G)+O(1))$. Here we went into the trouble of
estimating the error term more precisely so that Lemma \ref{lem:keq}.2
and some of its consequences can be stated more neatly.

Lemma \ref{lem:keq}.1 gives the following result.

\begin{theorem}\label{thm:inequi}
The number of graphs with logical depth at most $k$ does not exceed $\tower(k+\log^*k+2)$.
\end{theorem}

Notice two further consequences of Lemma~\ref{lem:keq}.

\break

\begin{theorem}\label{thm:inequi2}\mbox{}
\begin{bfenumerate}
\item
There are at most $\tower(k+\log^*k+3)$ pairwise inequivalent sentences 
about graphs of quantifier depth $k$.
\item
Every 
sentence $\Phi$ about graphs of quantifier depth $k$ has an equivalent
sentence $\Phi'$ with the same quantifier depth and length less than
$3\,\tower(k+\log^*k+2)^2$.
\end{bfenumerate}
\end{theorem}

\begin{proof}
Note that, if a sentence $\Phi$ has quantifier depth $k$,
then the set of all graphs on which $\Phi$ is true is the union
of some $\keq$-classes. 
Therefore, there are $2^{f(k)}$ and no more pairwise inequivalent sentences 
of quantifier depth $k$, where $f(k)$ is the number of $\keq$-classes.
Part 1 now follows from Lemma \ref{lem:keq}.1.
By the same reason every sentence $\Phi$ of quantifier depth $k$
is equivalent to the disjunction of sentences defining some $\keq$-classes.
By Lemma \ref{lem:keq}.2, such disjunction does not need to be longer than
$(f(k)+3)\tower(k+\log^*k+2)$. This proves Part~2.
\end{proof}

\subsection{Distinguishability vs.\ definability}

Given two non-isomorphic graphs $G$ and $H$,
we define $D(G,H)$ (resp.\ $W(G,H)$) to be the minimum of $D(\Phi)$ 
(resp.\ $W(\Phi)$) over all sentences $\Phi$ distinguishing $G$ from $H$.
Thus, $D(G,H)>k$ if and only if $G\keq H$.
Obviously, $D(G,H)=D(H,G)$. Also, $D(G,H)\le D(G)$ and $W(G,H)\le W(G)$.
It turns out that these inequalities are tight in the following sense. 

\begin{lemma}\label{lem:ddd}\mbox{}

\begin{bfenumerate}
\item
$D(G)=\max_{H\not\cong G}D(G,H)$.
\item
$W(G)=\max_{H\not\cong G}W(G,H)$.
\end{bfenumerate}
\end{lemma}

\begin{proof}
{\bflex 1.}
For each $H$
non-isomorphic to $G$ fix a sentence $\Phi_H$ that distinguishes $G$ from $H$
and has the minimum possible quantifier depth, i.e., $D(\Phi_H)=D(G,H)$.
Consider the sentence $\Phi\feq\bigwedge_{H\not\cong G}\Phi_H$. It distinguishes $G$
from each non-isomorphic $H$ and has quantifier depth $\max_H D(\Phi_H)$.
Therefore, $D(G)\le\max_H D(G,H)$ as wanted. An obvious drawback of this argument
is that the above conjunction over $H$ in $\Phi$ is actually infinite. However,
we have $D(\Phi_H)\le D(G)$ and there are only finitely many pairwise
inequivalent first-order sentences about graphs 
of bounded quantifier depth, see Theorem \ref{thm:inequi2} above. Thus
we can obtain
a legitimate finite sentence defining~$G$ by
removing from $\Phi$ duplicates up to logical equivalence.

{\bflex 2.}
Running the same argument, we have to ``prune''
the infinite conjunction $\bigwedge_{H\not\cong G}\Phi_H$, where $W(\Phi_H)=W(G,H)$.
Here we encounter a complication because there are infinitely many
inequivalent sentences of the same width. (Consider e.g.\ the
sentences  from  Example \ref{ex:LDW}.1.) 
However, Theorem \ref{thm:Dk}.1
in Section~\ref{s:wl} implies that for every $H$ we can additionally
require that the depth of $\Phi_H$ is at most, for example, 
$n^n+n$, where $n$ is the order of $G$. 
Now we can proceed as in Part 1 of the lemma.
\end{proof}

Lemma \ref{lem:ddd} stays true in any finite-variable logic,
any logic with bounded number of quantifier alternations, 
the logic with counting quantifiers, and any hybrid thereof.
We set $D^k(G,H)=\infty$ if $k$ variables do not suffice to
distinguish $G$ from~$H$.

\section{Ehrenfeucht games}\label{s:games}

Let $G$ and $H$ be graphs with disjoint vertex sets.
The \emph{$r$-round $k$-pebble Ehrenfeucht game on $G$ and $H$},
denoted by $\game_r^k(G,H)$, is played by
two players, Spoiler and Duplicator, to whom we may refer as he and
she respectively.
The players have at their disposal $k$ pairwise distinct
pebbles $p_1,\ldots,p_k$, each given in duplicate.
A {\em round\/} consists of a move of Spoiler followed by a move of
Duplicator. At each move Spoiler takes a pebble, say $p_i$, selects one of
the graphs $G$ or $H$, and places $p_i$ on a vertex of this graph.
In response Duplicator should place the other copy of $p_i$ on a vertex
of the other graph. It is allowed to move previously placed pebbles
to other vertices and place more than one pebble on the same vertex.

After each round of the game, for $1\le i\le k$ let $x_i$ (resp.\ $y_i$)
denote the vertex of $G$ (resp.\ $H$) occupied by $p_i$, irrespectively
of who of the players placed the pebble on this vertex. If $p_i$ is
off the board at this moment, $x_i$ and $y_i$ are undefined.
If after every of $r$ rounds the component-wise correspondence $(x_1,\ldots,x_k)$ to
$(y_1,\ldots,y_k)$ is a partial isomorphism from $G$ to $H$, this is
a win for Duplicator.  Otherwise the winner is Spoiler.
The following example should provide the reader with
a hint for the solution of the exercise suggested in Example~\ref{ex:LDW}.1.

\begin{example}\label{ex:game43}\rm
{\it Spoiler wins $\game_4^3(P_n,H)$ if $\Delta(H)\ge3$.}
Assume that $H$ contains no triangle because otherwise Spoiler wins
by pebbling its vertices.
Let $v$ be a vertex in $H$ of degree at least 3. Spoiler pebbles 3 neighbors
of $v$. Duplicator should pebble 3 distinct pairwise
non-adjacent vertices in $P_n$ for otherwise she loses the game. 
The distance between any two vertices pebbled
in $H$ is equal to 2. Unlike to this, some two vertices pebbled in $P_n$
(say, by pebbles $p_1$ and $p_2$) are at a larger distance. Spoiler moves $p_3$
to $v$. Duplicator is forced to violate the adjacency relation.
\end{example}

The particular case of $\game_r^k(G,H)$
in which the number of pebbles is the same as the number
of rounds, i.e., $k=r$, deserves a special attention. In this case,
the outcome of the game will not be affected if we prohibit moving pebbles 
from one vertex to another, that is, if we allow the players to play with each $p_i$
exactly once, say, in the $i$-th round.
We denote this variant of $\game_r^r(G,H)$ by $\game_r(G,H)$
and will mean it whenever the term \emph{Ehrenfeucht game} is used
with no specification.

\begin{lemma}\label{lem:distance}
Suppose that in the 3-pebble Ehrenfeucht game on $(G,H)$ some two
vertices $x,y\in V(G)$ at distance $n$ were selected so that their
counterparts $x',y'\in V(H)$ are at a strictly larger distance
(possibly infinity).
Then Spoiler can win in at most $\lceil\log n\rceil$ extra moves.
\end{lemma}

\begin{proof}
Spoiler sets $u_1=x$, $u_2=y$, $v_1=x'$, $v_2=y'$, and places
a pebble on the middle vertex $u$ in a shortest path from $u_1$ to $u_2$
(or either of the two middle vertices if $d(u_1,u_2)$ is odd).
Let $v\in V(H)$ be selected by Duplicator in response to $u$.
By the triangle inequality, we have $d(u,u_m)<d(v,v_m)$
for $m=1$ or $m=2$. For such $m$ Spoiler resets $u_1=u$, $u_2=u_m$,
$v_1=v$, $v_2=v_m$ and applies
the same strategy once again. In this way Spoiler ensures that
$d(u_1,u_2)<d(v_1,v_2)$ in each round. Eventually, unless Duplicator loses earlier,
$d(u_1,u_2)=1$ while $d(v_1,v_2)>1$, that is, Duplicator fails to
preserve adjacency.

To estimate the number of moves made, notice that initially
$d(u_1,u_2)=n$ and for each subsequent $u_1,u_2$ this distance
becomes at most $f(d(u_1,u_2))$, where $f(\ga)=(\ga+1)/2$.
Therefore the number of moves does not exceed the minimum $i$
such that $f^{(i)}(n)<2$. As $(f^{(i)})^{-1}(\gb)=2^i\gb-2^i+1$,
the latter inequality is equivalent to
$2^i\ge n$, which proves the bound.
\end{proof}

There is a rather clear connection between Spoiler's strategy
designed in the proof of Lemma \ref{lem:distance} and first-order
formula $\Delta'_n(x,y)$ in \refeq{ex:deltaa}.
We will see that, in some strong sense, $\game_r(G,H)$
corresponds to first-order logic, while $\game_r^k(G,H)$ corresponds
to its $k$-variable fragment.
In fact, every logic has its own corresponding game.

In the {\em $k$-alternation\/} variant of $\game_r(G,H)$
Spoiler is allowed to switch from one graph
to another at most $k$ times during the game, i.e., in at most $k$
rounds he can choose the graph other than that in the preceding round.

In the \emph{counting version} of the game $\game_r^k(G,H)$
Spoiler can make a \emph{counting move} consisting of two acts.
First, he specifies a set of vertices
$A$ in one of the graphs. Duplicator has to respond with a set of vertices $B$
in the other graph so that $|B|=|A|$ (if this is impossible,
she immediately loses). Second, Spoiler places a pebble $p_i$
on a vertex $b\in B$. In response Duplicator has to place the other copy
of $p_i$ on a vertex $a\in A$. It is clear that, any round with $|A|=1$
is virtually the same as a round of the standard game.

There is a general analogy between strategies allowing Spoiler to win
a game on $G$ and $H$ and first-order sentences distinguishing these graphs:
the former can be converted into the latter and vice versa so that
the duration of a game will be in correspondence to the quantifier
depth and the number of pebbles will be in correspondence to the number 
of variables.

\begin{theorem}[The Ehrenfeucht theorem and its variations]\label{thm:games}
Let $G$ and $H$ be non-isomorphic graphs.
\begin{bfenumerate}
\item {\rm(Ehrenfeucht \cite{Ehr}, Fra\"\i{}ss\'{e} \cite{Fra}% 
\footnote{It was Ehrenfeucht who formally introduced the game.
Prior to Ehrenfeucht, Fra\"\i{}ss\'{e} obtained virtually the same result
using an equivalent language of partial isomorphisms.})}
$D(G,H)$ equals the minimum $r$ such that Spoiler has a winning
strategy in $\game_r(G,H)$.
\item {\rm(Pezzoli \cite{Pez})}
$D_k(G,H)$ equals the minimum $r$ such that Spoiler has a winning
strategy in the $k$-alternation game $\game_r(G,H)$.
\item {\rm(Immerman \cite{Imm0}, Poizat \cite{Poi})}
$W(G,H)$ equals the minimum $k$ such that Spoiler has a winning
strategy in $\game_r^k(G,H)$ for some $r$.
\item {\rm(Immerman \cite{Imm0}, Poizat \cite{Poi})}
$D^k(G,H)$ equals the minimum $r$ such that Spoiler has a winning
strategy in $\game_r^k(G,H)$.
\item {\rm(Immerman and Lander~\cite{ILa})}
$\cw{G,H}$ equals the minimum $k$ such that Spoiler has a winning
strategy in the counting version of $\game_r^k(G,H)$ for some $r$. Furthermore,
if $k\ge\cw{G,H}$, then $\cd k{G,H}$ equals the minimum $r$ such that Spoiler has 
a winning strategy in the counting version of $\game_r^k(G,H)$.
\end{bfenumerate}
\end{theorem}

\noindent
We refer the reader to \cite[Theorem 6.10]{Imm} for the proof of
Parts~3--5. Part~1 follows from Part~4 in view of the facts that
$D(G,H)=\min_kD^k(G,H)$ and that any sentence $\Phi$ can be equivalently rewritten 
with the same quantifier depth $D(\Phi)$ and with 
use of at most $D(\Phi)$ variables. 

In view of Lemma \ref{lem:ddd},
the Ehrenfeucht theorem provides us with a powerful tool for
estimating the logical depth and width of graphs.
Consider, for instance, a path $P_n$.
Example \ref{ex:game43} and Lemma \ref{lem:distance}
are immediately translated into the upper bound $D^3(P_n)<\log n+3$.
On the other hand, a lower bound $D(P_n)\ge\log n-2$ follows from
the existence of a winning strategy for Duplicator in 
$\game_r(P_n,P_{n+1})$ whenever $r\le\lfloor\log n\rfloor-1$ 
(all details can be found in \cite[Theorem 2.1.3]{Spe}).

\section{The Weisfeiler-Lehman algorithm}\label{s:wl}

\emph{Graph Isomorphism} is the problem of
recognizing if two given graphs are isomorphic. The best
known algorithm (Babai, Luks, and Zemlyachenko \cite{babluk83}) 
takes time $2^{O(\sqrt{n\log n})}$, where $n$ denotes
the number of vertices in the input graphs. 
Particular classes of graphs for which Graph Isomorphism 
is solvable more efficiently are 
therefore of considerable interest. Somewhat surprisingly, a number of
important tractable cases are solvable by a combinatorially simple,
uniform approach, namely the multidimensional Weisfeiler-Lehman algorithm. 
The efficiency of this method depends much on the logical complexity of input graphs.

For the history of this approach to the graph isomorphism problem 
we refer the reader to \cite{Bab,CFI}.  
We will abbreviate \emph{$k$-dimensional Weisfeiler-Lehman algorithm}
by \emph{\WL k}. The \WL1\ is commonly known as
\emph{canonical labeling} or \emph{color refinement algorithm}. 
It proceeds in
rounds; in each round a coloring of the vertices of input graphs $G$ and
$H$ is defined, which refines the coloring of the previous round. The initial
coloring $C^0$ is uniform, say, $C^0(u)=1$ for all vertices $u\in V(G)\cup
V(H)$. In the $(i+1)$st round, the color $C^{i+1}(u)$ is defined to be a pair
consisting of the preceding color $C^{i-1}(u)$ and the multiset of colors
$C^{i-1}(w)$ for all $w$ adjacent to $u$. For example, $C^1(u)=C^1(v)$ iff $u$
and $v$ have the same degree.  To keep the color encoding short, after each
round the colors are renamed (we never need more than $2n$ color names\footnote{%
We do not need even more than $n$ because
appearance of the $(n+1)$th color indicates non-isomorphism.%
}). As the
coloring is refined in each round, it stabilizes after at most $2n$ rounds,
that is, no further refinement occurs. The algorithm stops once this
happens. If the multiset of colors of the vertices of $G$ is distinct from the multiset
of colors of the vertices of $H$, the algorithms reports that the graphs are
not isomorphic; otherwise, it declares them to be isomorphic. 
Disappointingly, the output is not always correct. 
The algorithm may report false positives, for example, if both
input graphs are regular with the same vertex degree.

Following the same idea, the $k$-dimensional version iteratively refines a
coloring of $V(G)^k\cup V(H)^k$. The initial coloring of a $k$-tuple $\baru$
is the isomorphism type of the subgraph induced by the vertices in $\baru$
(viewed as a labeled graph where each vertex is labeled by the positions in the
tuple where it occurs). Loosely speaking, the refinement step takes into account 
the colors of all neighbors of $\baru$ in the Hamming metric.  
Color stabilization is surely reached in
$
r<2n^k
$
rounds and, thus, the algorithm terminates in polynomial time for
fixed $k$.

Let us give a careful description of the \WL k for $k\ge2$.
Given an ordered $k$-tuple of vertices $\baru=(u_1,\ldots,u_k)\in V(G)^k$, we
define the \emph{isomorphism type} of $\bar u$ to be the pair
\begin{equation}\label{eq:isotype}
\isotype\baru=\Big(\setdef{(i,j)\in[k]^2}{u_i=u_j},\setdef{(i,j)\in[k]^2}
{\{u_i,u_j\}\in E(G)}\Big),
\end{equation}
where $[k]$ denotes the set $\{1,\ldots,k\}$.
If $w\in V(G)$ and $i\le k$, we let $\baru^{i,w}$ denote the result
of substituting $w$ in place of $u_i$ in $\baru$. 

The \emph{$r$-round \WL k}
takes as an input two graphs $G$ and $H$ and purports to decide
if $G\cong H$. The algorithm performs 
the following operations with the set $V(G)^k\cup V(H)^k$.

\tlt{Initial coloring.}
The algorithm assigns each $\baru\in V(G)^k\cup V(H)^k$ color
$\wl k0\baru=\isotype\baru$ (in a suitable encoding).

\tlt{Color refinement step.}
In the $i$-th round each $\baru\in V(G)^k$ is assigned color
$$
\wl ki\baru=\Big(\wl k{i-1}\baru,
\msetdef{\left(\wl k{i-1}{\baru^{1,w}},\ldots,\wl k{i-1}{\baru^{k,w}}\right)}{w\in V(G)}\Big)
$$
and similarly with each $\baru\in V(H)^k$.

Here $\mso\ldots\msc$ denotes a multiset. In a weaker \emph{count-free} 
version of the algorithm, this notation will be interpreted as a set.
Let
$$
\wl krG=\msetdef{\wl kr{\baru}}{\baru\in V(G)^k}.
$$

\tlt{Computing an output.}
The algorithm reports that
$G\not\cong H$ if
\begin{equation}\label{eq:decision}
\wl krG\ne\wl krH
\end{equation}
and that $G\cong H$ otherwise.

In the above description we skipped an important implementation detail.
In order 
to prevent increasing the length of $\wl ki\baru$ at the exponential rate,
we arrange colors of all $k$-tuples 
of $V(G)^k\cup V(H)^k$ in the lexicographic order and replace each color with 
its number before every refinement step.

Furthermore, let
$$
\diagwl krG=\msetdef{\wl kr{u^k}}{u\in V(G)},
$$
where $u^k$ denotes the $k$-tuple $(u,\ldots,u)$.

\begin{lemma}\label{lem:diagineq}
In both the standard and the count-free versions of the \WL k, inequality
\begin{equation}\label{eq:ineq1}
\diagwl krG\ne\diagwl krH
\end{equation}
implies \refeq{eq:decision}, which in its turn implies
\begin{equation}\label{eq:ineq3}
\diagwl k{r+k-1}G\ne\diagwl k{r+k-1}H.
\end{equation}
\end{lemma}

\begin{proof}
Consider the standard version; the analysis of the count-free case
is similar (and even simpler). 
By the \emph{equality type} of a $k$-tuple $\baru$ we mean the first
component of \refeq{eq:isotype}.
Note that $k$-tuples with different equality types never have
the same color. Therefore, $\wl krG$ and $\wl krH$ are different 
iff they are different on some class of $k$-tuples with the same equality type.
This proves the first implication.

On the other hand, suppose that \refeq{eq:decision} holds. 
Let $E$ be an equality type
on which $\wl krG$ and $\wl krH$ differ. Note that each $\baru$ in $E$
contributes color $\wl kr{\baru}$ (a certain number of times)
to color $\wl k{r+k-1}{a^k}$. Moreover, the sum of the contributions over
all vertices $a$ is the same for every $\baru\in E$.
It follows that, if a color has different multiplicities in $\wl krG$ and $\wl krH$,
its ``traces'' occur different number of times in $\diagwl k{r+k-1}G$
and $\diagwl k{r+k-1}H$, and hence these multisets are distinct.
\end{proof}

As it is easily seen, if $\phi$ is an isomorphism from $G$ to $H$, then for all $k$, $i$,
and $\baru\in V(G)^k$ we have $\wl ki\baru=\wl ki{\phi(\baru)}$.
This shows that for isomorphic input graphs the output is always correct.
If input graphs are non-isomorphic and the dimension $k$ is not big enough,
the algorithm can erroneously report isomorphism. 
A criterion for the optimal choice of the dimension is obtained by
Cai, F\"urer, and Immerman \cite{CFI}, who discovered a connection between
the Weisfeiler-Lehman algorithm and the logical complexity of graphs
via the Ehrenfeucht game (for the color refinement algorithm this
was done by Immerman and Lander \cite{ILa}). 
The success of the standard version of the algorithm depends
on distinguishability of the input graphs in the logic with counting
quantifiers, while the count-free version is in the same way related
to the standard first-order logic.

Referring to the \WL k below, we will \textbf{always} assume $k\ge1$
for the standard version of the algorithm and $k\ge2$
for its count-free version (we can exclude the case of $k=1$, whose
analysis differs by some details, as the count-free \WL 1 is of no interest:
note that it is unable to distinguish between two graphs of order $n$
without isolated and universal vertices).

Given numbers $r$, $l$, and $k\le l$, graphs $G$, $H$, and
$k$-tuples $\baru\in V(G)^k$,
$\barv\in V(H)^k$, we use notation $\game^l_r(G,\baru,H,\barv)$
to denote the $r$-round $l$-pebble Ehrenfeucht game on $G$ and $H$
with initial configuration $(\baru,\barv)$, that is, the game starts
on the board with  $k$ already pebbled pairs $(u_i,v_i)$.
If the initial configurations is not a partial isomorphism,
Duplicator loses $\game^l_r(G,\baru,H,\barv)$ whatever $r\ge0$.
The following lemma is a key element of our analysis.

\begin{lemma}[Cai, F\"urer, and Immerman \cite{CFI}]\label{lem:game}
Let $\baru\in V(G)^k$ and $\barv\in V(H)^k$.
\begin{bfenumerate}
\item
Equality
\begin{equation}\label{eq:coleq}
\wl kr\baru=\wl kr\barv
\end{equation}
holds for (the standard version of) the \WL k iff Duplicator has a winning
strategy in the counting version of $\game_r^{k+1}(G,\baru,H,\barv)$.
\item
Equality \refeq{eq:coleq}
holds for the count-free version of the \WL k iff Duplicator has a winning
strategy in (the standard version of) $\game_r^{k+1}(G,\baru,H,\barv)$.
\end{bfenumerate}
\end{lemma}

\begin{proof}
We prove only Part 2 (Part 1 is proved in detail in \cite[Theorem 5.2]{CFI}).
We proceed by induction on $r$. The base case $r=0$ is straightforward
by the definitions of the initial coloring and the game. Assume that
the proposition is true for $r-1$ rounds.

Let $x_i$ and $y_i$ denote the vertices in $G$ and $H$ respectively
marked by the $i$-th pebble pair. Assume \refeq{eq:coleq}
and consider the Ehrenfeucht game on $G$, $H$
with initial configuration $(x_1,\ldots,x_k)=\baru$ and
$(y_1,\ldots,y_k)=\barv$. First of all, this configuration is non-losing
for Duplicator since \refeq{eq:coleq} implies that
$\isotype\baru=\isotype\barv$. Further, Duplicator can survive in
the first round. Indeed, assume that Spoiler in this round selects
a vertex $a$ in one of the graphs, say in $G$. Then Duplicator selects
a vertex $b$ in the other graph $H$ so that
$\wl k{r-1}{\baru^{i,a}}=\wl k{r-1}{\barv^{i,b}}$ for all $i\le k$.
In particular,
$\isotype{\baru^{i,a}}=\isotype{\barv^{i,b}}$ for all $i\le k$.
Along with $\isotype\baru=\isotype\barv$, this implies that
$\isotype{\baru,a}=\isotype{\barv,b}$.
Assume now that in the second round Spoiler removes $j$-th pebble,
$j\le k$. Then Duplicator's task in the rest of the game is
essentially to win $\game^{k+1}_{r-1}(G,\baru^{j,a},H,\barv^{j,b})$.
Since $\wl k{r-1}{\baru^{j,a}}=\wl k{r-1}{\barv^{j,a}}$,
Duplicator succeeds by the induction assumption.

Assume now that \refeq{eq:coleq} is false. It follows that
$\wl k{r-1}\baru\ne\wl k{r-1}\barv$ (then Spoiler has a winning
strategy by the induction assumption) or there is a vertex $a$ in one
of the graphs, say in $G$, such that for every $b$ in the other graph 
$H$ we have $\wl k{r-1}{\baru^{j,a}}\ne\wl k{r-1}{\baru^{j,b}}$
for some $j=j(b)$. In the latter case Spoiler in his first move
places the $(k+1)$-th pebble on $a$. Let $b$ be the vertex selected
in response by Duplicator. In the second move Spoiler will remove
the $j(b)$-th pebble, which implies that
the players essentially play
$\game_{r-1}^{k+1}(G,\baru^{j,a},H,\barv^{j,b})$ from now on.
By the induction assumption, Spoiler wins.
\end{proof}

\begin{lemma}\label{lem:diagD}
Equality $\diagwl krG=\diagwl krH$ is true for the standard (resp.\ count-free)
version of the \WL k iff Duplicator has a winning strategy in
the counting (resp.\ standard) version of $\game_{r+1}^{k+1}(G,H)$.
\end{lemma}

\begin{proof}
We consider the standard version of the algorithm; 
the proof for the count-free version is very similar.
If the multisets $\diagwl krG$ and $\diagwl krH$ are not equal,
Spoiler has a winning strategy in the counting game
$\game_{r+1}^{k+1}(G,H)$.
In the first round he makes a counting move that forces pebbling
$a\in V(G)$ and $b\in V(H)$ so that $\wl kr{a^k}\ne\wl kr{b^k}$.
The remainder 
of the game is equivalent to the counting game $\game_{r}^{k+1}(G,a^k,H,b^k)$,
where Spoiler has a winning strategy by Lemma~\ref{lem:game}.

If the multisets $\diagwl krG$ and $\diagwl krH$ are equal, Duplicator
is able to play the first round so that $\wl kr{a^k}=\wl kr{b^k}$
for the pebbled vertices $a$ and $b$. She wins the remaining game
again by Lemma~\ref{lem:game}.
\end{proof}

We say that the $r$-round \WL k\ \emph{works correctly for a graph $G$}
if its output is correct on all input pairs $(G,H)$ (here $H$ may have
any order, not necessary the same as $G$).

\begin{theorem}\label{thm:WLD}\mbox{}
The $r$-round \WL k\ works correctly for $G$ if
$$
k\ge \cw G-1\ \ \mbox{and}\ \ r \ge \cd{k+1}{G}-1
$$
and only if
$$
k\ge \cw G-1\ \ \mbox{and}\ \ r \ge \cd{k+1}{G}-k.
$$

The same holds true for the count-free $r$-round \WL k\
and the standard logic (without counting).
\end{theorem}

\begin{proof}
If $G\cong H$, the output is correct in any case.
Suppose that $G\not\cong H$. 
By Lemma \ref{lem:diagineq}, inequality \refeq{eq:ineq1} is
a sufficient condition for the output being correct while
\refeq{eq:ineq3} is a necessary condition for this.
The theorem now follows from Lemma \ref{lem:diagD},
the Ehrenfeucht theorem (Theorem \ref{thm:games}.4,5),
and Lemma~\ref{lem:ddd}.1 along with its counting version.
\end{proof}

By Theorem \ref{thm:WLD}, $k\ge \cw G-1$ is both a sufficient and a necessary
condition for a successful work of the \WL k\ on all inputs $(G,H)$.
As we already discussed, the number of rounds can be taken $r=O(n^k)$.
Therefore, Graph Isomorphism is solvable in polynomial time for any
class of graphs $C$ with $\cw G=O(1)$ for all $G\in C$.
This applies to any class of graphs embeddable into a fixed surface
and any class of graphs with bounded treewidth (see Section~\ref{ss:classes}).

Sometimes the Weisfeiler-Lehman algorithm gives us even better
result, namely the solvability of the isomorphism problem by a parallel algorithm 
in polylogarithmic time. The concept of polylogarithmic
parallel time is captured by the complexity class NC and its refinements:
$$
\nc{}=\textstyle\bigcup_i\nc i\textrm{ and }
\nc i\subseteq\ac i\subseteq\tc i\subseteq\nc{i+1},
$$
where \nc i\ consists of functions computable by circuits of polynomial
size and depth $O(\log^in)$, \ac i\ is an analog for circuits with unbounded
fan-in, and \tc i\ is an extension of \ac i\ allowing threshold gates.  As
it is 
well known \cite{KRa}, \ac i\ consists of exactly those functions computable
by a CRCW PRAM with polynomially many processors in time $O(\log^in)$. 
Grohe and Verbitsky \cite{GVe} point out that the $r$-round \WL k\
(resp.\ its count-free version) is implementable in \tc1\ (resp.\ \ac1)
as long as $k=O(1)$ and $r=O(\log n)$. If combined with Theorem \ref{thm:WLD},
this gives us the following result.
\begin{theorem}\label{thm:GVe}
Let $k\ge2$ be a constant.
\begin{bfenumerate}
  \item
    Let $C$ be a class of graphs $G$ with
    $\cd {k}G=O(\log n)$. Then Graph Isomorphism for $C$ is solvable in \tc1.
  \item
    Let $C$ be a class of graphs $G$ with $D^{k}(G)=O(\log n)$.
    Then Graph Isomorphism for $C$ is solvable in \ac1.
  \end{bfenumerate}
\end{theorem}
We will see applications of Theorem \ref{thm:GVe} in Section~\ref{ss:classes}.

Suppose that $k\ge W(G)-1$ and that we do not know a priori any bounds
for $D^{k+1}(G)$. How large has $r$ to be taken in order to ensure that
the $r$-round \WL k\ works correctly for $G$?
An answer is given by an important concept of \emph{color stabilization}
that was already discussed in the beginning of this section.
We will regard $\wwl kr$ as a partition of $V(G)^k\cup V(H)^k$. Let $R$ be the minimum
number for which $\wwl kR=\wwl k{R-1}$. Of course, it is enough to check the
condition \refeq{eq:decision} for $r=R$; it cannot change for bigger $r$.
Since each $\wwl kr$ is a refinement of $\wwl k{r-1}$, we have $R\le v(G)^k+v(H)^k$.
In fact, we are able to prove a bit more delicate claim: 
The Weisfeiler-Lehman algorithm can be 
terminated as soon as $\wwl kr$ stabilizes at least within $V(G)^k$.

To make this more precise, we introduce some notation.
Denote the restriction of the partition $\wwl kr$ to $V(G)^k$
by $\wwl kr_G$. Let $\stabi kG$ be the smallest number $s$
such that $\wwl ks_G=\wwl k{s-1}_G$. 
Note that $\stabi kG$ is an individual combinatorial parameter of
a graph $G$, not depending on $H$ (we may think that the \WL k is run
on a single graph $G$, which is actually a quite
meaningful \emph{canonization mode} of the algorithm).

We now state practical termination rules for the \WL k.

\begin{description}
\item[\it Rule 1]
Once $\wl krG\ne\wl krH$, terminate and report non-isomorphism.
\item[\it Rule 2]
Once $r=\stabi kG$ and $\wl krG=\wl krH$, terminate and report isomorphism.
\end{description}

Let us argue that these rules are sound for both
versions of the algorithm. Suppose that Rule~2 is invoked. Thus $\wwl
kr_G=\wwl k{r-1}_G$ and $\wl krG=\wl krH$. By the latter equality we also have $\wl k{r-1}G=\wl
k{r-1}H$. It follows that in the $r$-th round the algorithm achieves
a proper color refinement on neither $G$ nor $H$. Thus,
the partition $\wwl kr$ has been stabilized on $V(G)^k\cup V(H)^k$ and
the soundness of Rule~2 follows.

\begin{theorem}\label{thm:WLS}
\mbox{}

\begin{bfenumerate}
\item
The $r$-round \WL k\ recognizes non-isomorphism of $G$ and $H$ if 
$$k\ge\cw{G,H}-1\ \ \mbox{and}\ \ r \ge \stabi kG.$$
\item
The $r$-round \WL k\ works correctly for $G$ if 
$$k\ge\cw G-1\ \ \mbox{and}\ \ r \ge \stabi kG.$$
\item
Both claims hold true for the count-free version of the algorithm 
and the standard logic (with no counting).
\end{bfenumerate}
\end{theorem}

We have seen that good bounds for the logical complexity of graphs
imply efficiency of the Weisfeiler-Lehman algorithm on these graphs.
Now we will get a couple of noteworthy facts on the logical complexity
as a consequence of our analysis of the algorithm.

\begin{theorem}\label{thm:Dk}
Let $G$ be a graph of order $n$.

\begin{bfenumerate}
\item
If $G$ is distinguishable from another graph $H$ in the $\ell$-variable logic,
then $D^{\ell}(G,H)\le n^{\ell-1}+\ell-2$.
\item
If $G$ is definable in the $\ell$-variable logic,
then $D^{\ell}(G)\le n^{\ell-1}+\ell-2$.
\end{bfenumerate}
\end{theorem}

\begin{proof}
 Let $k=\ell-1$. 
Comparing the sufficient conditions for the correctness
of the $r$-round \WL k\ given by Theorem \ref{thm:WLS} and the necessary
conditions given by Theorem \ref{thm:WLD}, we have
$D^{k+1}(G,H)\le\stabi kG+k$ provided $k\ge W(G,H)-1$ and
$D^{k+1}(G)\le\stabi kG+k$ provided $k\ge W(G)-1$.
For the former claim we need also the fact, actually
established in the proof of Theorem \ref{thm:WLD}, that the count-free $r$-round \WL k\ 
is able to recognize non-isomorphism of $G$ and $H$ only if
$k\ge W(G,H)-1$ and $r \ge D^{k+1}(G,H)-k$.
It remains to notice that $\stabi kG\le n^k-1$.
\end{proof}

A somewhat weaker bound $D^{\ell}(G)\le n^{\ell}+\ell+1$ follows from the work
of Dawar, Lindell, and Weinstein \cite[Corollary 4]{DLW}.

\section{Worst case bounds}\label{s:worstcase}

\subsection{Classes of graphs}\label{ss:classes}

Here we overview known bounds for the logical depth and width
for natural classes of graphs. Several interesting definability effects
can be observed even when we focus on so simple graphs as trees.
This class is considered at the beginning of this section (and will
be further discussed in Sections \ref{s:average}
and \ref{s:best}). We will see that many results about trees admit generalization
to graphs with bounded treewidth. 
We further consider
planar graphs. Then we briefly discuss more general cases of graphs
embeddable into a fixed surface and graphs
with an excluded minor, as well as a few sporadic results on other classes.

\subsubsection{Trees} 

The following result is based on Edmonds' algorithm, that dates back to the sixties
(see, e.g., \cite{CBo}), and its logical interpretation is due to 
Immerman and Lander~\cite{ILa}.

\begin{theorem}\label{thm:cwtrees}\mbox{}
\begin{bfenumerate}
\item
The color refinement algorithm succeeds in recognizing isomorphism of trees.
Consequently, 
$\cw{T,T'}\le2$ for every two non-isomorphic trees $T$ and~$T'$.
\item
$\cw{T}\le2$ for every tree~$T$.
\end{bfenumerate}
\end{theorem}

\begin{proof}
{\bflex 1.}
As in Section \ref{s:wl}, let $C^r$ denote the coloring appearing 
after the $r$-th refinement. Let $N_r(v)$ denote the set of all
vertices at the distance at most $r$ from a vertex $v$.
It is not hard to see that, if $v$ is an arbitrary vertex in
a tree $T$, then the subtree spanned by $N_r(v)$ is,
up to isomorphism, reconstructible from $C^r(v)$.
Let $v$ and $v'$ be arbitrary vertices in trees $T$ and $T'$.
If $T\not\cong T'$, we have $C^r(v)\ne C^r(v')$ at latest
for $r$ one greater than the smaller of the eccentricities of $v$ and $v'$.
Therefore, the color refinement algorithm distinguishes between any two
non-isomorphic trees. The second statement of Part~1 follows 
by Lemma \ref{lem:game}.1
and Theorem~\ref{thm:games}.5.

{\bflex 2.}
To obtain the desired definability result, we use the equality 
$
\cw T=\max_{H\not\cong T}\cw{T,H},
$
which is an analog of Lemma \ref{lem:ddd}.2
(with a much simpler proof as graphs of different orders are
distinguishable with a single counting quantification).
Thus, it suffices to prove that $\cw{T,H}\le2$ whenever $H\not\cong T$.
Suppose that $H$ is not a tree for otherwise we are done by Part~1. 
Also, as it was just mentioned, we can suppose 
that both $T$ and $H$ have $n$ vertices.

Assume first that $H$ has a connected component $T'$ which is a tree.
Note that $T'\not\cong T$ because $T'$ has less than $n$ vertices.
Let $v\in V(T)$ and $v'\in V(T')$. Run the color refinement algorithm
on input $(T,H)$.
As in the proof of Part 1 we have $C^n(v)\ne C^n(v')$
because the coloring $C^n$ on $T'$ is the same as if the algorithm was run
on $T'$ instead of $H$. Therefore, $T$ and $H$ are distinguishable with 2
variables in the counting logic.

If none of the connected components of $H$ is a tree,
then $H$ has at least $n$ edges. Since $T$ has exactly $n-1$
edges, $H$ and $T$ have distinct multisets of vertex degrees
and, hence, are distinguishable by a sentence with 2 counting
quantifiers.
\end{proof}

The proof of Theorem \ref{thm:cwtrees}
gives us only a linear upper bound $\cd2{T}=O(n)$
for a tree of order $n$.
We can get a speed-up if we allow more variables.

\begin{theorem}\label{thm:cdtrees}
For every tree $T$ on $n$ vertices we have
$$
\cd3T< 3\,\log n.
$$
\end{theorem}

\begin{proof}
By an analog of Lemma \ref{lem:ddd}.1 for the counting logic and 
Theorem \ref{thm:games}.5, we have to show that Spoiler is able
to win the counting game $\game_r^3(T,T')$ with some $r<3\,\log n$
for any graph $T'$ non-isomorphic to $T$.
Suppose that $T'$ has the same order $n$.
If $T'$ is disconnected, Spoiler wins (even without counting moves)
by Lemma \ref{lem:distance}. If $T'$ is connected and has a cycle,
then $T$ and $T'$ have distinct multisets of vertex degrees.
Therefore, we will suppose that $T'$ is a tree too.

Every tree $T$ has a single-vertex \emph{separator}, that is, a vertex $v$ such that
no branch of $T-v$ has more than $n/2$ vertices;
see, e.g., Ore \cite[Chapter 4.2]{Ore}.
The idea of Spoiler's strategy is to pebble such a vertex and
to force further play on some non-isomorphic branches of $T$ and $T'$,
where the same strategy can be applied recursively.

\begin{figure}
\centerline{\scalebox{0.65}{\includegraphics{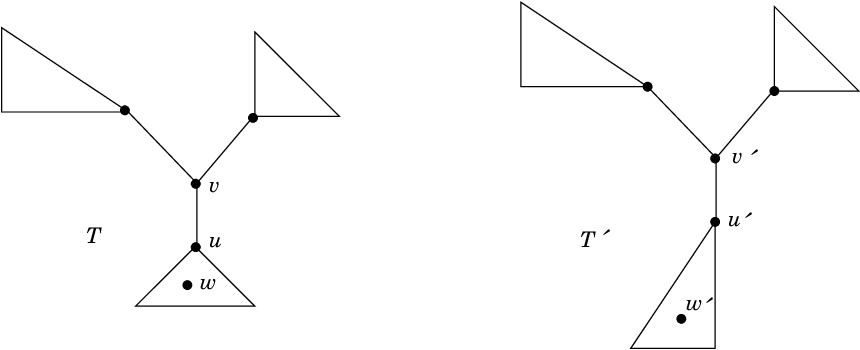}}}
\caption{A separator strategy of Spoiler.}
\label{fig:tree1}
\end{figure}

Thus, in the first round Spoiler pebbles a separator $v$ in $T$
and Duplicator responds with a vertex $v'$ somewhere in $T'$.
The component of $T-v$ containing a neighbor $u$ of $v$ will be
denoted by $T_{vu}$ and considered a rooted tree with the root at $u$.
A similar notation will apply also to $T'$.
In the second round Spoiler makes a counting move and ensures that
$u\in N(v)$ and $u'\in N(v')$ are pebbled so that the
rooted trees $T_{vu}$ and $T'_{v'u'}$ are non-isomorphic,
see Fig.~\ref{fig:tree1}.
The next goal of Spoiler is to force pebbling adjacent vertices $v_1$ and $u_1$
in $T_{vu}$ and adjacent vertices $v'_1$ and $u'_1$ in $T'_{v'u'}$ so that
$T_{v_1u_1}\not\cong T'_{v'_1u'_1}$, $V(T_{v_1u_1})\subset V(T_{vu})$, and $v(T_{v_1u_1})\le v(T_{vu})/2$.
Once this is done, the same will be repeated recursively.

To make the transition from $T_{vu}$ to $T_{v_1u_1}$, Spoiler follows three rules.

{\it Rule 1.}
If $T_{vu}$ has a branch $T_{ux}$ for some $x\in N(u)\setminus\{v\}$ 
such that $v(T_{ux})\le v(T_{vu})/2$
and the number of branches isomorphic to $T_{ux}$ 
is different for $T_{vu}$ and $T'_{v'u'}$,
then Spoiler makes a counting move and forces pebbling 
such $x$ and $x'\in N(u')\setminus\{v'\}$
so that $T_{ux}\not\cong T'_{u'x'}$. 
The latter two branches will serve as $T_{v_1u_1}$ and $T'_{v'_1u'_1}$.
If no such branch is available,
Spoiler pebbles a separator $w$ of $T_{vu}$. Note that Duplicator
is forced to respond with a vertex $w'$ in $T'_{v'u'}$.
Otherwise we would have $\dist(w,u)=\dist(w,v)-1$ while $\dist(w',u')=\dist(w',v')+1$.
Therefore, some distances among the three pebbled vertices
would be different in $T$ and in $T'$ and Spoiler could win
in less than $\log v(T_{vu})+1$ moves by Lemma \ref{lem:distance}.

{\it Rule 2.}
If $T$ differs from $T'$
by some branch $T_{wx}$ (having a different number of
occurrences in $T'-w'$) that does not contain $u$, Spoiler
makes a counting move with the pebble released from $v$ and
forces pebbling such $x$ in $T$ and some $x'$ in $T'$
so that $T_{wx}\not\cong T'_{w'x'}$. These branches will
serve as $T_{v_1u_1}$ and $T'_{v'_1u'_1}$. (It is possible that
$T'_{w'x'}$ contains $u'$ or $T_{wx}$ contains $u$ but then the distances among $u,w,x$ 
are not all equal to the distances among $u',w',x'$ and Spoiler
quickly wins.)

\begin{figure}
\centerline{\scalebox{0.65}{\includegraphics{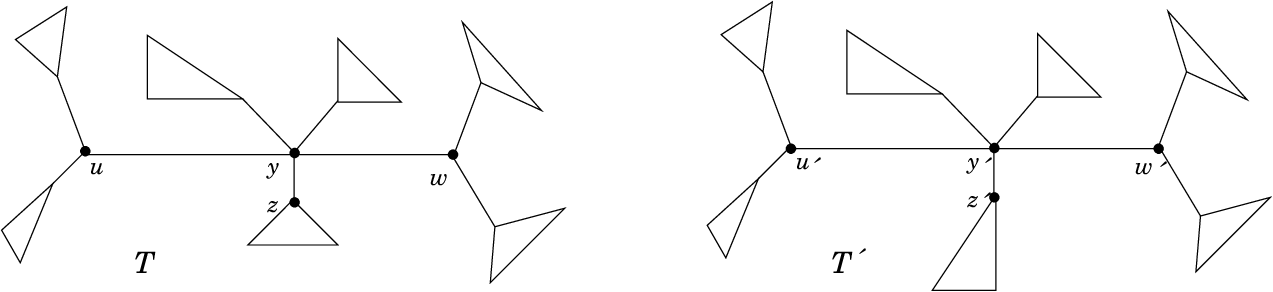}}}
\caption{Rule 3 invoked.}
\label{fig:tree2}
\end{figure}

{\it Rule 3.}
Denote the branch of $T_{vu}-w$ containing $u$ by $T_{w,u}$
(and similarly for $T'$). If Rule 2 is not applicable, then
$T_{w,u}$ and $T'_{w',u'}$ are non-isomorphic
(where an isomorphism would need to respect two pairs of designated vertices,
namely $u$ and $u'$ as well as the neighbors of $w$ and $w'$).
Assume the harder subcase that $\dist(u,w)=\dist(u',w')$. 
When Spoiler pebbles a vertex $y$
on the path from $u$ to $w$ by moving the pebble from 
$v$, Duplicator is forced to
pebble the corresponding vertex $y'$ on the path from $u'$ to $w'$.
It is easy to see that a vertex $y\ne w$ can be chosen so that
$T-y$ and $T'-y'$ differ by branches containing neither $u$ and $u'$
nor $w$ and $w'$. 
Let Spoiler pebble such $y$ as close to $w$ as possible.
Note that $y\ne u$ because otherwise Rule 1 was applicable.
Now Spoiler can make a counting move with the pebble released from $w$
to force pebbling $z$ and $z'$ for which 
\begin{equation}\label{eq:Tnoniso}
T_{yz}\not\cong T'_{y'z'},
\end{equation}
see Fig.~\ref{fig:tree2}. This complies with the goal of finding
new $T_{v_1u_1}$ and $T'_{v'_1u'_1}$ because
\begin{equation}\label{eq:halforder}
v(T_{yz})<v(T_{w,u})\le v(T_{vu})/2.
\end{equation}

In fact, Duplicator could try to prevent the fulfillment of \refeq{eq:Tnoniso}
by forcing a choice of $z'$ such that $T'_{y'z'}$ would contain either $u'$
or $w'$. In the former case Spoiler could win by using differences
between the distances among $u,y,z$ and among $u',y',z'$.
In the latter case \refeq{eq:Tnoniso} would anyway be true because
$T_{yz}$ and $T'_{y'z'}$ would have different orders.
Indeed, since Rule 2 was not applicable, the choice of $y$ ensures
that the branches of $T-y$ and $T'-y'$ containing $w$ and $w'$
are isomorphic. Thus, we would have $v(T'_{y'z'})=v(T_{y,w})\ge v(T_{vu})/2$
(where $T_{y,w}$ denotes the branch of $T-y$ containing $w$)
while $v(T_{yz})$ is strictly smaller by~\refeq{eq:halforder}.

 Given $T_{vu}$, Spoiler finds a new distinguishing branch $T_{v_1u_1}$
in 3 rounds in the worst case. Also, 2 rounds suffice to win the game
once the current subtree $T_{vu}$ has at most 4 vertices. 
The number of transitions from the initial branch of order at most
$\lceil n/2\rceil$ to one with at most $4$ vertices is bounded by 
$\log \lceil n/2\rceil-2$ 
because $v(T_{vu})$
becomes twice smaller each time. Routine calculations (and
Lemma~\ref{lem:distance}) imply the desired bound
on the length of the game.
\end{proof}

The definability of trees in a finite-variable counting logic
within logarithmic quantifier depth can also be derived from 
a work by Etessami and Immerman \cite{EIm}, which also implies that
counting quantifiers are here not needed as long as the maximum
vertex degree is bounded by a constant.

Curiously, Theorem \ref{thm:cdtrees} 
sheds some new light on the history of isomorphism testing for
trees. The first record of this history was made by Edmonds, who showed that
the problem is solvable in linear time (see Theorem \ref{thm:cwtrees}). 
Ruzzo \cite{Ruz} found an \ac1 algorithm under the condition that
the vertex degrees of input trees are at most logarithmic in the number
of vertices.
Miller and Reif \cite{MRe} established an \ac1 upper bound unconditionally.
They wrote \cite[page 1128]{MRe}: 
``No polylogarithmic parallel algorithm was previously known for
isomorphism of unbounded-degree trees.''
However, the 2-dimensional Weisfeiler-Lehman algorithm has been discussed
in the literature at least since 1968 (e.g., \cite{WLe}) and, as we now
see by combining Theorem \ref{thm:cdtrees} with Theorem \ref{thm:GVe}.1, 
this algorithm does the job for arbitrary trees in \nc2, i.e., in parallel time $O(\log^2n)$ !

To complete this historical overview, we have to mention a result by
Lindell \cite{Lin} who showed that isomorphism of trees is recognizable
in logarithmic space. 
Though Lindell's result is best possible (see Jenner et al.\ \cite{JKMT}),
the solvability of the problem by so simple and natural procedure 
as the Weisfeiler-Lehman algorithm still remains a noteworthy fact.

Note that $\cd2{P_n}=\frac n2-O(1)$
(it is not hard to see that the color refinement algorithm
requires at least $\frac n2-O(1)$ rounds to distinguish between
$P_n$ and the disjoint union of $P_{n-3}$ and $C_3$).
Thus, Theorem \ref{thm:cdtrees} shows a jump from
linear to logarithmic quantifier depth 
when the number of variables is increased just by 1.
Such width-depth
trade-offs were observed and studied by F\"urer~\cite{Fue}.

Theorem \ref{thm:cwtrees} says that 2 variables and counting quantifiers
suffice to define any tree.
Moreover, we could well manage without counting quantifiers but
then we would need to have $\Delta(T)+1$ variables.
A simple example of a star, where $W(K_{1,m})=m+1$, shows
that a smaller number is not enough.
The following bound is a variant of a result by Immerman
and Kozen \cite{IKo}, who consider definability of trees
represented by an asymmetric child-parent relation between vertices.

\begin{theorem}\label{thm:wtrees}
$W(T)\le\Delta(T)+1$ for any tree $T$ with the exceptions of $T\in\{P_1,P_2\}$.
\end{theorem}

The logical depth of a tree can be bounded in terms of the maximum
degree and the order. 

\begin{theorem}[Bohman et al.\ \cite{BFL*}]\label{thm:dtrees}\mbox{}
\begin{bfenumerate}
\item
For every tree $T$ of order $n$ with maximum vertex degree $\Delta(T)\ge9$ we have
$$
D(T)\le\of{\frac{\Delta(T)}{2\log(\Delta(T)/2)}+3}\log n+\frac{3\Delta(T)}2+O(1).
$$
\item
Let $D(n,d)$ be the maximum of $D(T)$ over all
trees with $n$ vertices and maximum degree at most $d=d(n)$.
If both $d$ and $\log n/\log d$ tend to infinity, then
 $$
 D(n,d)=\left(\frac12+o(1)\right)\, d\,\frac{\log n}{\log d}.
 $$
\end{bfenumerate}
\end{theorem}

The upper bounds on $D(T)$ comes from Spoiler's strategy similar to
that of the proof of Theorem~\ref{thm:cdtrees}, that is, Spoiler
pebbles a separator $v$ of the given tree $T$ and then tries to
restrict the game to one of the components of $T-v$. Informally
speaking, the worst case scenario for Spoiler is when $T-v$ has $d$
components of order about $n/d$ of two different isomorphism types,
each occurring half of the time. Then Spoiler may need around $d/2$
extra moves to restrict game to a component of $T-v$ (if the
components of the counterpart $T'-v'$ are of these two types but with
different multiplicities). Thus, roughly, Spoiler ``reduces'' the order
by factor $d$ using $d/2$ moves, which gives the heuristic for the
bound of Theorem~\ref{thm:dtrees}.2. The optimality of this bound is given by a
recursive construction of a tree $T$ (and another tree $T'\not\cong
T$), where at each recursion step we glue together about $d/2$ trees
of two different isomorphism types at a common root.

\subsubsection{Graphs of bounded treewidth}

Informally speaking, the \emph{treewidth} of a graph tells us
to which extent the graph is representable as a tree-like structure.
This concept appeared in the Robertson-Seymour theory and, aside of
its theoretical importance, found a lot of applications in design of
algorithms on graphs. We do not go into any detail here, referring instead
to the books \cite{Die} and \cite{DFe} that may serve as introductions to, 
respectively, the structural theory of graphs and the algorithmic applications.

It happens quite often that techniques applicable to trees can be
extended to graphs whose treewidth is bounded by a constant.
In particular, this is true for the definability parameters.

\begin{theorem}\label{thm:btw}\mbox{}
\begin{bfenumerate}
\item
{\rm (Grohe and Mari\~no \cite{GMa})}
If a graph $G$ has treewidth $k$, then $\cw G\le k+2$.
\item
{\rm (Grohe and Verbitsky \cite{GVe})}
If a graph $G$ on $n$ vertices has treewidth $k$, then 
$$
\cd{4k+4}G<2(k+1)\log n+8k+9.
$$
Consequently, isomorphism of graphs whose treewidth does not exceed $k$ 
is recognizable by the $(4k+3)$-dimensional
Weisfeiler-Lehman algorithm in $\tc1$.
\end{bfenumerate}
\end{theorem}

The last claim in the theorem follows a general paradigm provided by
Theorem \ref{thm:GVe}.1: A low quantifier depth
implies solvability of the isomorphism problem in NC.
Prior to \cite{GVe}, for graphs with bounded treewidth only 
polynomial-time isomorphism test of Bodlaender \cite{Bod} was known.
Very recently Das, Tor\'an, and Wagner \cite{DTW} put the problem in the
complexity class LOGCFL.

Like Theorem \ref{thm:cdtrees}, the proof of Theorem \ref{thm:btw}.2 
is based on separator techniques. In general, a set $X\subset V(G)$
will be called a \emph{separator} for graph $G$ if any component of $G\setminus X$
has at most $n/2$ vertices. It is well known \cite{RSe} that all graphs
of treewidth $k$ have separators of size $k+1$.

\subsubsection{Planar graphs}\label{sss:planar}

The separator techniques in the study of logical complexity of graphs 
were introduced by Cai, F\"urer, and Immerman \cite{CFI}, who derived
a bound $\cw G=O(\sqrt n)$ for planar graphs from the known fact
\cite{LTa} that every planar graph of order $n$
has a separator of size $O(\sqrt n)$. 
In fact, this result is a particular case of Theorem \ref{thm:btw}.1
because planar graphs have treewidth bounded by $5\sqrt5\,n$; see \cite[Proposition 4.5]{AST}.
Later Grohe \cite{Gro1}
proved that $\cw G$ for all planar $G$ is actually bounded
by a constant. 

Without counting quantifiers we cannot have 
any nontrivial upper bound for the logical depth in terms of the order
of a graph
as long as a class under consideration contains all trees.
However, some natural classes of planar graphs admit such bounds.
A plane drawing of a graph is called \emph{outerplanar}
if all the vertices lie on the boundary of the outer face.
\emph{Outerplanar graphs} are those planar graphs having an outerplanar drawing.
The treewidth of any outerplanar graph is at most 2.
As it is well known (see, e.g., \cite{Har}), any outerplanar graph is representable 
as a tree of its biconnected components. Note also that an outerplanar graph
is biconnected iff it has a Hamiltonian cycle and that such a graph can
be geometrically viewed as a dissection of a convex polygon.

\begin{theorem}[Verbitsky \cite{Ver,Ver_planar}]\label{thm:planar}\mbox{}
\begin{bfenumerate}
\item 
If $G$ is a biconnected outerplanar graph of order $n$, then
$
D(G) < 22\,\log n + 9.
$
\item
For a 3-connected planar graph $G$ of order $n$ we have 
$
D^{15}(G)<11\,\log n+45
$.
\end{bfenumerate}
\end{theorem}

Part 2 shows another case when Theorem \ref{thm:GVe} is applicable.
It gives an \ac1 isomorphism test for 3-connected planar graphs and,
by a known reduction of  Miller and Reif \cite{MRe}, for the whole class of planar graphs.
This complexity bound for the planar graph isomorphism is not new; 
it follows from the \ac1 isomorphism test
for embeddings designed in \cite{MRe} and the \ac1 embedding algorithm
in \cite{RRe}. As a possible advantage of the Weisfeiler-Lehman approach, note that it
is combinatorially much simpler and more direct. In particular, we
do not need any embedding procedure here.
The best possible complexity bound for the planar graph isomorphism
is recently obtained by Datta et al.\ \cite{DLNTW} who design a logarithmic-space
algorithm for this problem.

Theorem \ref{thm:planar}.1 is proved in \cite{Ver} and is based on the existence
of a 2-vertex separator in any outerplanar graph. The possibility to avoid counting
quantifiers relies on certain rigidity of biconnected outerplanar graphs.
The latter is related to the following geometric fact: Any such graph 
has a unique, up to homeomorphism, outerplanar drawing.

The case of 3-connected planar graphs is much more complicated because
the smallest separators in such graphs can have about $\sqrt n$ vertices
(such examples can be obtained by adding a few edges to the grid graph $P_m\times P_m$).
The proof of Theorem \ref{thm:planar}.2 in \cite{Ver_planar}
exploits a strong rigidity property of 3-connected planar graphs:
By the Whitney theorem (see, e.g., \cite{MTh}), they have a unique, up to homeomorphism,
embedding into the sphere. An embedding can be represented
as a purely combinatorial structure, called a \emph{rotation system}
(see \cite{MTh}), to which one can extend the concepts of definability,
isomorphism, the Ehrenfeucht game etc. Defining rotation systems
is a simpler business because they admit a kind of coordinatization
and hence an analog of the halving strategy from Lemma \ref{lem:distance} 
is available for Spoiler. The most essential ingredient of the proof
of Theorem \ref{thm:planar}.2 is a strategy for Spoiler in the Ehrenfeucht
game on graphs allowing him to simulate the Ehrenfeucht game on 
the corresponding rotation systems.

\subsubsection{Graphs with an excluded minor}\label{sss:exclminor}

No graph with treewidth $h$ has $K_{h+2}$ as a minor.
The class of graphs embeddable into a closed 2-dimensional surface $S$
is closed under minors and, as follows from the Robertson-Seymour 
Graph Minor Theorem, no graph from this class contains a minor of $K_h$ 
for some $h=h(S)$. Extending his earlier work on graphs embeddable
into a fixed surface \cite{Gro2}, Grohe \cite{Gro4} recently announced a proof
that, if a graph $G$ does not contain $K_h$ as a minor, then
$\cw G$ is bounded by a constant $c=c(h)$. The case of $h=5$ is treated
in detail in~\cite{Gro3}.

Alon, Seymour, and Thomas \cite{AST} proved that, if a graph $G$ of order $n$
does not contain a $K_h$ as a minor, then it has a separator of size 
at most $h^{3/2}\sqrt n$. Using this result, for all connected graphs
with this property one can prove \cite{Ver} that
$D(G)=O(h^{3/2}\sqrt n)+O(\Delta(G)\log n)$.

\subsubsection{Other classes of graphs}\label{sss:others}

A graph is \emph{strongly regular} if all its vertices have equal degrees
and, for some $\lambda$ and $\mu$, each pair of adjacent vertices
has exactly $\lambda$ common neighbors and each pair of non-adjacent vertices
has exactly $\mu$ common neighbors. Non-isomorphic graphs with the same
order, degree, and parameters $\lambda$ and $\mu$ are standard examples
of a failure of the \WL 2\ algorithm. Babai studies the isomorphism problem for this class
in \cite{Bab2}. His individualization-and-refinement technique translates
into a bound $\cw G\le2\sqrt n\log n$ for all strongly regular graphs
of a sufficiently large order $n$ with the exception for the disjoint
unions of complete graphs and their complements (for which we have $\cd{}G\le3$).
Further improvements are obtained by Spielman~\cite{Spi}.

Evdokimov, Ponomarenko, and Tinhofer \cite{EPT00} undertake an analysis
of the 3-dimensional WL algorithm on the classes of cographs,
interval graphs, and even directed path graphs (the latter class
extends the class of interval graphs and contains also all ptolemaic
graphs, in particular, trees). It follows from \cite{EPT00} that
$\cw G \le 4$ for all $G$ in any of these classes.
The boundedness of $\cw G$ for interval graphs follows also from 
the paper of Laubner \cite{Lau},
who uses purely logical methods (while Evdokimov et al. develop
an algebraic approach that, in fact, originates from the seminal work by 
Weisfeiler and Lehman).

Grohe \cite{Gro5} proves that $\cw G=O(1)$ for all chordal line graphs. 
On the other hand, he shows that there are chordal graphs
with $\cw G=\Omega(n)$ and the same holds true for line graphs.
The latter result is obtained by a reduction to the graphs with $\cw G=\Omega(n)$
constructed by Cai, F\"urer, and Immerman \cite{CFI} (cf.\ Theorem \ref{thm:CFI} below).
Note that the Cai-F\"urer-Immerman graphs are regular of degree 3, where the regularity
can be traded for the bipartiteness after a slight modification.

\subsection{General case}\label{s:general}

\subsubsection{Identification problem}\label{sss:identif}

Recall that 
$$
\cw{G,H}\le W(G,H)\le D(G,H)\textrm{ and }\cw{G,H}\le \cd{}{G,H}\le D(G,H).
$$
If we are motivated by the graph isomorphism problem, it is quite natural
to focus on these parameters under the assumption that 
$G$ and $H$ have the same order 
(even without saying that $\cd{}{G,H}=1$ otherwise).
Distinguishing a graph $G$ from all non-isomorphic $H$ of \emph{the same order}
is sometimes called \emph{identification problem}.
In particular, we would like to determine or estimate
the maximum of $D(G,H)$ (resp.\ $\cd{}{G,H}$) as a function 
of $n=v(G)=v(H)$. 
Equivalently, what is the minimum $r=r(n)$ such that
Spoiler has a winning strategy in $\game_r(G,H)$ for all non-isomorphic $G$ and $H$
of order~$n$?

By taking disjoint unions of complete and empty graphs,
it is easy to find $G$ and $H$ with $D(G,H)\ge(n+1)/2$.
Bounding $\cd{}{G,H}$ from below is much more subtle issue.
Using a nice nontrivial argument, Cai, F\"urer, and Immerman \cite{CFI}
came up with a linear lower bound.

\begin{theorem}[Cai, F\"urer, and Immerman \cite{CFI}]\label{thm:CFI}
For infinitely many $n$ there are non-isomorphic graphs $G$ and $H$
both of order $n$ such that $\cw{G,H}\ge c\,n$, where $c$ is a positive constant.
\end{theorem}

The calculation of Pikhurko et al.\ 
\cite[Section 7.5]{PVVarxiv} shows that one can take $c = 0.00465$.

Let us turn to upper bounds. Suppose that $G\not\cong H$ and $v(G)=v(H)=n$.
Before reading further, the reader might try
to improve the trivial bound $D(G,H)\le n$ at least somewhat.
It may be seen as a curious observation that $D(G,H)\le n-1$ 
follows from the Harary version of the Ulam 
Reconstruction Conjecture, open for a long time, 
claiming that non-isomorphic graphs of equal orders
have different sets of vertex-deleted subgraphs.

One solution of this exercise, giving $D(G,H)<n-\frac14\log n$, is to
apply the Erd\H os-Szekeres bound on Ramsey numbers. It implies that 
every graph $G$ of large order $n$ contains a homogeneous set of more than
$\frac12\log n+\frac14\log\log n$ vertices.
 Spoiler pebbles the complement of such a set $S$ in $G$.
Suppose that the unpebbled set is independent (otherwise we can play
on the complementary graphs). If Duplicator is lucky, she manages
to pebble the complement to an independent set $S'$ in $H$
so that $G\setminus S\cong H\setminus S'$. Identifying the pebbled
parts, Spoiler compares the number of vertices in $S$ and in $S'$
with the same neighborhood. These numbers
cannot be identical for $G$ and $H$ and, by $v(G)=v(H)$, Spoiler can
demonstrate this using at most $(|S|+1)/2$ further moves in one of the
graphs.

After this warm-up, we can state an almost optimal bound.

\begin{theorem}[Pikhurko, Veith, and Verbitsky \cite{PVV}]\label{thm:PVV}
For every two non-iso\-mor\-phic graphs $G$ and $H$ of the same order $n$ we have
$D(G,H)\le (n+3)/2$.
\end{theorem}

\subsubsection{General bounds for the logical depth and width}\label{ss:generaldw}

In the case of the counting logic, 
Theorem \ref{thm:CFI} provides us with infinitely many graphs $G$
for which $\cd{}G\ge\cw G > 0.00465\,n$. 
As usually, $n$ denotes the order of a graph.
An upper bound easily follows from Theorem \ref{thm:PVV}: 
we have $\cd{}G\le0.5\,n+1.5$ for all $G$.
Though this bound does not use the power of counting quantifiers at all,
we are not aware of any better bound.

Consider the standard first-order logic (without counting).  At the
first sight, everything is clear here. Indeed, the general upper bound
$W(G)\le D(G)\le n+1$ is attained, even for the width, by the complete
graph $K_n$ and by the empty graph $\compl{K_n}$.  However, these are
the only two extremal graphs. In other words, $D(G)\le n$ for all $G$
with exception of $G\in\{K_n,\compl{K_n}\}$.  As $K_n$ and
$\compl{K_n}$ are the most symmetric graphs, this observation suggests
two problems. The first one is to prove a better bound for a
class of graphs with restrictions on the automorphism group. The
second is to obtain, for as small as possible $l=l(n)$, an explicit or
algorithmic description of all order-$n$ graphs whose logical depth
(resp.\ width) exceeds $l$. We start with the first problem.

\begin{definition}\label{def:twins}\rm
Let $u$, $v$, and $s$ be three vertices and $s\notin\{u,v\}$.
We say that $s$ \emph{separates} $u$ and $v$ if
$s$ is adjacent to exactly one of the two vertices.
Furthermore, we call $u$ and $v$ \emph{twins} if 
no $s$ separates $u$ and $v$ (or, equivalently, if the transposition of $u$ and $v$
is an automorphism of the graph). A graph is called \emph{twin-free}
if it has no twins.
\end{definition}

\begin{theorem}[Pikhurko, Veith, and Verbitsky~\cite{PVV}]\label{thm:irred}
If $G$ is twin-free, then $D_1(G)\le(n+5)/2$.
\end{theorem}

Theorem \ref{thm:irred} cannot be improved to a sublinear bound.
Indeed, consider $mP_4$, the disjoint union of $m$ copies of $P_4$.
As it is easily seen, $mP_4$ is twin-free and $D(mP_4)\ge D(mP_4,(m+1)P_4)>m$
(the reader is welcome to play $\game_m(mP_4,(m+1)P_4)$ on Duplicator's side).
No sublinear improvement is possible even with counting quantifiers:
the graphs constructed by Cai, F\"urer, and Immerman in Theorem \ref{thm:CFI}
are twin-free.

We prove Theorem \ref{thm:irred} based on Lemma \ref{lem:ddd}.1 and
the Ehrenfeucht Theorem (Theorem \ref{thm:games}.1). That is, we design
a strategy allowing Spoiler to win $\game_r(G,H)$ for any $H\not\cong G$, where
$r=\lfloor (n+5)/2\rfloor$. As an important additional feature of the strategy,
Spoiler will alternate between the graphs only once. By Theorem \ref{thm:games}.2,
this shows that our bound holds even for the logic with only one quantifier
alternation (as it is indicated by the subscript in Theorem~\ref{thm:irred}).

\begin{definition}\label{def:weaksieve}\rm
Let $X\subset V(G)$.
Given two vertices $u,v\in V(G)\setminus X$, we call them \emph{$X$-similar}
and write $u\xeq v$ if $u$ and $v$ are inseparable by any vertex in $X$, i.e.,
if $N(u)\cap X=N(v)\cap X$.

Now, let $y\notin X$. We say that $X$ \emph{sifts out} $y$ if for
every $y'\notin X$ the relation $y\xeq y'$ implies $y'=y$
(in other words, the vertex $y$ is uniquely identified by its
adjacencies to $X$).
Let $\sieve(X)$ consist of all $x\in X$ and all $y$ sifted out by $X$.
We call $X$ a \emph{sieve}\footnote{%
Babai \cite{Bab2} uses sieves under the name \emph{distinguishing sets}.} 
if $\sieve(X)=V(G)$.
Furthermore, $X$ is called a \emph{weak sieve} if $\sieve(\sieve(X))=V(G)$.
\end{definition}

Consider the Ehrenfeucht game on non-isomorphic $G$ and $H$ and assume
that $X$ is a sieve in $G$. Let Spoiler pebble all vertices of $X$.
We leave to the reader to verify that Spoiler can win in at most 2 more moves.
We now describe a more advanced \emph{Weak Sieve Strategy}.

\begin{lemma}\label{lem:weaksieve} 
If $X$ is a weak sieve in $G$, then Spoiler 
is able, for any $H\not\cong G$, to win $\game_r(G,H)$ with $r\le|X|+3$.
Moreover, he does not need to jump from one graph
to the other more than once during the game.
\end{lemma}

\begin{proof}
 First, Spoiler selects all of $X$. Let
$X'\subset V(H)$ be the Duplicator's reply. Assume that Duplicator has
not lost yet. For the notational simplicity let us identify $X$ and
$X'$ so that $V(G)\cap V(H)=X=X'$ and the player's moves coincide on
$X$.  Let $Y=\sieve(X)$ in $G$ and $Y'= \sieve(X')$ in~$H$.

It is not hard to see that Spoiler wins in at most two extra moves
unless the following holds. For any $y\in Y\setminus X$ there is a
$y'\in Y'\setminus X$ (and vice versa) such that $N(y)\cap
X=N(y')\cap X$. Moreover, this bijective correspondence between
$Y$ and $Y'$ establishes an isomorphism between $G[Y]$ and
$G'[Y']$. 

Suppose that this is the case and identify $Y$ with $Y'$. Let
$Z=V\setminus Y$ and $Z'= V'\setminus Y$. Let $z\in Z$ and define
 $$
 W'_{z}=\setdef{z'\in Z'}{N(z')\cap Y = N(z)\cap Y}.
 $$
 
If $W'_z=\emptyset$, Spoiler wins in at most two moves. First, he
selects $z$. Let Duplicator reply with $z'$. Assume that $z'\in Z'$ for
otherwise she has already lost. As the
neighborhoods of $z,z'$ in $Y$ differ, Spoiler can demonstrate this by
picking a vertex of $Y$. If $|W'_z|\ge 2$, then Spoiler selects any
two vertices in $W'_z$ and wins with at most one more move, as
required.

Hence, we can assume that for any $z$ we have $W'_z=\{f(z)\}$ for some
$f(z)\in Z'$. 
Since each vertex in $Z$ is sifted out by $Y$, the function $f$ is injective.
If $f(Z)\ne Z'$, Spoiler easily wins in two moves. Suppose, therefore,
that $f\function Z{Z'}$ is a bijection. As $G\not\cong H$,
the mapping $f$ does not preserve the adjacency relation between some
$y,z\in Z$. Now, Spoiler selects both $y$ and $z$. Duplicator cannot
respond with $f(y)$ and $f(z)$; by the definition of $f$ Spoiler can
win in one extra move.
\end{proof}

Theorem \ref{thm:irred} immediately follows from Lemma \ref{lem:weaksieve}
and the next lemma.

\begin{lemma}
Any twin-free 
graph $G$ on $n$ vertices has a weak sieve $X$ with $|X|\le(n-1)/2$.
\end{lemma}

\begin{proof}
Given $X\subset V(G)$, let $\calC(X)$ denote the partition of $\compl X=V(G)\setminus X$ 
into $\xeq$-equivalence classes.
Starting from $X=\emptyset$, we repeat
the following procedure. As long as there exists $u\in\compl X$
such that $|\calC(X\cup\{u\})|>|\calC(X)|$, we move $u$ to $X$.
As soon as there is no such $u$, we arrive at $X$ which is
{\em $\calC$-maximal\/}, that is,
$|\calC(X\cup\{u\})|\le|\calC(X)|$ for any $u\in\compl X$. Note that
$|\calC(X)|\ge|X|+1$ because this inequality is true at the beginning and is
preserved in each construction step. Using also the
inequality $|X|+|\calC(X)|\le n$, we conclude that  
$|X|\le(n-1)/2$.

We now prove that the $X$ is a weak sieve.
Suppose, to the contrary, that $u$ and $v$ are distinct $\sieve(X)$-similar
vertices in $Z=V(G)\setminus\sieve(X)$. Since $G$ has no twins,
these vertices are separated by some $s$. We cannot have $s\in\sieve(X)$
by the definition of $\sieve(X)$-similarity. Thus 
$s\in Z$. Let $C_1$ be the class in $\calC(X)$ including $\{u,v\}$
and $C_2$ be the class in $\calC(X)$ containing $s$. Since $s\notin \sieve(X)\setminus X$,
the class $C_2$ has at least one more element in addition to $s$.
If $C_1\ne C_2$, moving $s$ to $X$ splits up $C_1$ and does not
eliminate $C_2$. If $C_1=C_2$, moving $s$ to $X$ splits up this class
and splits up or does not affect the others. In either case
$|\calC(X)|$ increases, giving a contradiction.
\end{proof}

The proof of Theorem \ref{thm:irred} is complete. This theorem was
significantly extended in \cite{PVV} giving some progress on the
second research problem stated above. In particular, it was shown that
one can efficiently check whether or not $D(G)\le(n+5)/2$ for the
input graph $G$ of order $n$ and, if this is not true, then one can
efficiently compute the exact value of $D(G)$. Also, the same holds for~$W(G)$.

This result is interesting in view of the fact that
algorithmic computability of the logical depth and width of a graph,
even with no efficiency requirements, is unclear.
A reason for this is that the question if a given first-order sentence defines 
some graph is known to be undecidable~\cite{PSV}.

The upper bound of $\frac12n+O(1)$ can be improved if we impose
a restriction on the maximum vertex degree.

\begin{theorem}[Pikhurko, Veith, and Verbitsky~\cite{PVV}]\label{thm:bdeg}
Let $d\ge2$. Let $G$ be a graph of order $n$ with no
isolated vertex and no isolated edge.
If $\Delta(G)\le d$, then
$$
D_1(G) < c_d\, n+d^2+d+4
$$
for a constant $c_d=\frac12-\frac14d^{-2d-5}$.
\end{theorem}
Theorem \ref{thm:bdeg} aims at showing a constant $c_d$ strictly less
than $1/2$ rather than at attempting to find the optimum $c_d$.
In the case of $d=2$, which is simple and included just for uniformity,
an optimal bound is $D_1(G)\le n/3+O(1)$.
Without the assumption that $G$ has no isolated vertex and edge, the theorem
does not hold for any fixed $c_d<1/2$. A counterexample is provided by
the disjoint union of isolated edges.
Even under the stronger assumption that $G$ is connected, Theorem \ref{thm:bdeg}
still does not admit any sublinear improvement: the Cai-F\"urer-Immerman
graphs in Theorem \ref{thm:CFI} are connected and have maximum degree~3.

\section{Average case bounds}\label{s:average}

In Section \ref{s:worstcase} we investigated the maximum values of
logical parameters over graphs of order $n$.  Now we want to know its
typical values.  A natural setting for this problem is given by the
Erd\H{o}s-R\'enyi model of a random graph $\rpgraph$. The latter is a
random graph on $n$ vertices where every two vertices are connected by an edge
with probability $p$ independently of the other pairs. A particularly
important case is $\rgraph$, when we have the uniform distribution on
all graphs on a fixed set of $n$ vertices.  Whenever we say that for a
random graph of order $n$ something happens \emph{with high
probability} (abbreviated as \emph{whp}), 
we mean a probability approaching $1$ as $n\to\infty$.

\subsection{Bounds for almost all graphs}

\subsubsection{Logic with counting}

We begin with 
a simple but useful observation about the color refinement
algorithm described at the beginning of Section \ref{s:wl}:
If the coloring of a graph stabilizes with all color classes becoming singletons,
it can be considered a canonical vertex ordering.
It turns out that this happens for almost all graphs.
This result can be used to estimate the logical complexity of almost all graphs, 
in particular, to show that
almost surely $\cw\rgraph=2$ (Immerman and Lander~\cite{ILa}).

\begin{theorem}\label{thm:BES}\mbox{}

\begin{bfenumerate}
\item
{\rm (Babai, Erd\H{o}s, and Selkow \cite{BES})}
2 color refinements split a random graph $\rgraph$ into color classes which are 
singletons with probability more than $1-1/\sqrt[7]n$, for all large enough $n$.
Consequently, $\cd2{\rgraph}\le4$ with this probability.
\item
{\rm (Babai and Ku\v{c}era \cite{BKu})}
3 color refinements split a random graph $\rgraph$ into color classes which are 
singletons with probability more than $1-1/2^{cn}$, for a constant $c>0$ and 
all large enough $n$.
Consequently, $\cd2{\rgraph}\le5$ with this probability.
\end{bfenumerate}
\end{theorem}

The logical conclusions made in Theorem \ref{thm:BES}
are based on the necessity part of Theorem \ref{thm:WLD}. 
It suffices to notice that, 
once the color refinement splits the vertex set of an input graph $G$
into singletons, one extra round of the algorithm suffices to distinguish
$G$ from any non-isomorphic graph.

Next, we are going to show that the upper bound of
Theorem~\ref{thm:BES}.1 is best possible.
Let $C^r_G$ denote the coloring of the vertex set of a graph $G$
produced by the color refinement procedure in $r$ rounds.

\begin{lemma}\label{lem:1colref}
Whp for $G=\rgraph$ there exists a non-isomorphic
graph $H$ on $V=V(G)$ such that $C^{2}_G(x)=C^{2}_H(x)$ for every
vertex $x\in V$.
\end{lemma}

\begin{proof}
Let $G$ be a typical graph of a sufficiently large order $n$.
In particular, we assume that $G$ satisfies Theorem~\ref{thm:BES}.1
and that $|\deg x-n/2|\le m$ for every vertex $x$ of $G$, where we can
take e.g. $m=\sqrt{(n\log n)/2}$ by a simple application of Chernoff's
bound (see also \cite[Corollary~3.4]{Bol}).
By the Pigeonhole Principle,
there is a set $U$ of $u=\lceil n/(2m+1)\rceil$ vertices
all having the same degree.

Another property of the random graph $G$ that we assume is that every set
$X\subset V$ of size $u$ contains distinct vertices $w,x,y,z$ with
$wx,yz\in E(G)$ and $xy,zw\not\in E(G)$. 
Indeed, let us fix a $u$-set $X\subset V$
and estimate the probability that it violates this property. 
One can find at least ${u\choose 2}/5$ edge-disjoint
$4$-cycles inside the complete graph on $X$. 
(For example, picking up
cycles one by one arbitrarily, we get enough of them by the 
well-known fact
that a $C_4$-free graph on $u$ vertices has $O(u^{3/2})$ edges,
see, e.g., \cite[Chapter 25.5]{AZi}.) 
For each $4$-cycle on vertices $x_1,x_2,x_3,x_4$ in this order, at
least one of the relations
$x_1x_2,x_3x_4\in E(G)$ and $x_2x_3,x_4x_1\not\in E(G)$ should be false, this
having probability $15/16$. By the edge-disjointness, these events for
different selected cycles are mutually independent. Hence $X$ 
violates the desired property with probability is at most $(15/16)^{{u\choose 2}/5} =
o({n\choose u}^{-1})$. Since there are ${n\choose u}$ candidates for 
a bad set $X$, the probability that it exists is $o(1)$, giving the required.

Hence, the equidegree set $U$ contains vertices $w,x,y,z$ with
$wx,yz\in E(G)$ and $xy,zw\not\in E(G)$. Let $H$ be obtained from $G$
by removing edges $wx,yz$ and adding edges $xy,zw$. This operation
preserves the degree of every vertex as well as the multiset of degrees
of its neighbors, that is,  $C^{2}_G(v)=C^{2}_H(v)$ for every
vertex $v\in V$. 

Suppose that $G$ and $H$ are isomorphic. Any isomorphism $f$ must
preserve the $C^{2}$-colors. Since $C^{2}$-classes are all
singletons, $f$ has to be the identity map on $V(G)=V(H)$. But then
the adjacency between, e.g., $w$ and $x$ is not preserved, a
contradiction. The lemma is proved.\end{proof}

Given a typical $G=\rgraph$, let $H$ be a
graph satisfying Lemma~\ref{lem:1colref}.
Thus, the 2-round color refinement fails to distinguish between $G$ and $H$.
By the sufficiency part of Theorem \ref{thm:WLD}, we have $\cd2{G}>3$.
As an alternative proof, the reader can design a winning strategy
for Duplicator in the counting game $\game^2_3(G,H)$.
Together with Theorem \ref{thm:BES}.1, this bound gives us the exact value
$\cd2{\rgraph}$.

\begin{theorem}\label{thm:cdrgraph}
Whp $\cd2{\rgraph}=4$.
\end{theorem}

We always have $\cd{}{G}\le\cd2{G}$ and,
on the other hand, $\cd{}{G}\le2$ implies $\cd2{G}\le2$ because any definition with
quantifier depth 2 can be rewritten with using only 2 variables.
It follows from Theorem \ref{thm:cdrgraph} that
$3\le\cd{}{\rgraph}\le4$ whp.
Unfortunately, we could not decide whether the typical value of $\cd{}\rgraph$ is $3$ or
$4$, which seems to be an interesting question.

\newpage

\subsubsection{Logic without counting}

\begin{theorem}[Kim et al.~\cite{KPSV}]\label{thm:wdrgraph}
Fix an arbitrarily slowly increasing function $\omega=\omega(n)$. Then
we have whp that
\begin{multline*}
\log n-2\log\log n+\log\log\e+1-o(1)\le W(\rgraph)\le\\
\le D_1(\rgraph)\le\log n-\log\log n+\omega.
\end{multline*}
\end{theorem}

We first prove the lower bound.

\begin{definition}\label{def:alice}\rm
For an integer $k\ge 1$, we 
say that a graph $G$ has the {\em $k$-extension property\/} if, 
for every two disjoint $X,Y\subset V(G)$ with $|X\cup Y| \le k$, 
there is a vertex $z\notin X\cup Y$ 
adjacent to all $x\in X$ and non-adjacent to all $y\in Y$.
\end{definition}

\begin{lemma}\label{lem:alice1}
If both $G$ and $H$ have $k$-extension property, then $W(G,H)\ge k+2$.
\end{lemma}

\begin{proof}
By Theorem \ref{thm:games}.3 it suffices to design a strategy allowing
Duplicator to survive in $\game^{k+1}(G,H)$ arbitrarily long.
Suppose that Spoiler puts pebble $p$ on a new position $v$ in one of the graphs,
say, $G$.
Let $X$ (resp.\ $Y$) denote the set of pebbled vertices in $H$
whose counter-parts in $G$ are adjacent (resp.\ non-adjacent) to $v$.
Duplicator moves the other copy of $p$ to a vertex $z$ with the given
adjacencies to $X\cup Y$ whose existence is guaranteed
by the $k$-extension property.
\end{proof}

\begin{lemma}\label{lem:alice2}
Let $\epsilon>0$ be a real constant.
Then the $k$-extension property 
holds for $\rgraph$ whp
for any $k\le\log n-2\log\log n+\log\log\e-\epsilon$.
\end{lemma}

\begin{proof} Let $n$ be large.
Any particular $X$ and $Y$ with $|X\cup Y|=k$ falsify the $k$-extension
property with probability $(1-2^{-k})^{n-k}$. Since the number of such
pairs is ${n\choose k}2^{k}$, a random graph $\rgraph$ does not
have the $k$-extension property with probability at most
$$
{n\choose k}2^{k}(1-2^{-k})^{n-k}\le n^{k}(1-2^{-k})^n\le
\exp\ofc{k\ln n-n2^{-k}}.
$$
The former inequality is true only if $k\ge4$ but this makes no problem
because the $k$-extension property implies itself for all smaller
values of parameter $k$. Since the function $f(x)=x\ln n-n2^{-x}$
is monotone, the $k$-extension property fails with the probability 
bounded from above by
\begin{multline*}
\exp\ofc{f(\log n-2\log\log n+\log\log\e-\epsilon)}=\\
=\exp\ofc{(\ln 2)\, (-2^{\epsilon}+1+o(1))\log^2 n}=o(1),
\end{multline*}
as it was claimed.
\end{proof}

Fix $\epsilon>0$. Let $n$ be sufficiently large and set $k=\lfloor\log n-2\log\log n+\log\log\e-\epsilon\rfloor$.
By Lemma \ref{lem:alice2}, $G=\rgraph$ has the $k$-extension property
whp. Let $H$ be a graph which also possesses the $k$-extension
property and is non-isomorphic to $G$. The existence of such a graph follows
also from Lemma \ref{lem:alice2}: Given $G$, let $H=\rgraph$ be another, independent copy
of a random graph. It should be only noticed that $H\cong G$ with probability
at most $n!2^{-{n\choose 2}}=o(1)$. By Lemma \ref{lem:alice1}, we have
$$
W(G)\ge W(G,H)\ge k+2>\log n-2\log\log n+\log\log\e+1-\epsilon,
$$
thereby proving the lower bound of Theorem~\ref{thm:wdrgraph}.

To prove the upper bound, we employ the \emph{Weak Sieve Strategy}
that was designed in Section \ref{ss:generaldw}. Lemma \ref{lem:weaksieve}
allows us to estimate the parameter $D_1(G)$ by the size of a
weak sieve existing in $G$.
The upper bound of Theorem \ref{thm:wdrgraph} follows from
Lemmas \ref{lem:wsievergraph} below, that gives us a good enough bound
for the size of a weak sieve in a random graph.
The paper~\cite{KPSV}
states a slightly weaker upper bound than that in
Theorem~\ref{thm:wdrgraph} (namely, $\omega=C\log\log
\log n$ there). The current more precise estimate is due to 
Joel Spencer (unpublished).

\begin{lemma}\label{lem:wsievergraph}
Fix an arbitrarily slowly increasing function $\omega=\omega(n)$.
Then whp $\rgraph$ has a weak sieve of size at most
$\log n-\log\ln n+\omega$.
\end{lemma}

\begin{proof}
We will consider a random graph $\rgraph$ on an $n$-vertex set $V$.
Fix $X\subset V$ with $|X|=\log n-s$, where $s=\log\ln n-\omega$.
We generate $\rgraph$ in two stages. 

\smallskip

{\it Stage 1:} reveal the edges between $X$ and $V\setminus X$
(needless to say, each such edge appears with probability 1/2 independently
of the others). Our goal at this stage is to show that $\sieve(X)$ is large
whp.

A fixed $y\in V\setminus X$ is sifted out by $X$ with probability
$$
(1-2^{-|X|})^{n-|X|-1}=\exp\ofc{-2^s(1+o(1))}=n^{-2^{-\omega}(1+o(1))}=n^{-o(1)}.
$$
By linearity of expectation
$$
\expect{|\sieve(X)\setminus X|}=(n-|X|)n^{-o(1)}=n^{1-o(1)}.
$$
We can now apply the martingale techniques to show that whp
$|\sieve(X)\setminus X|$ is concentrated near its mean value.
More precisely, we need the following estimate:
\begin{equation}\label{eq:sieve1}
\prob{|\sieve(X)\setminus X|<\expect{|\sieve(X)\setminus X|}-2\lambda\sqrt{n-|X|}}<
\e^{-\lambda^2/2}
\end{equation}
for any $\lambda>0$, where $\prob A$ denotes the probability of an
event $A$.

To prove it, consider the probability space consisting
of all functions $g\function{V\setminus X}{2^X}$. Define a random variable $L$
on this space by setting $L(g)$ to be equal to the number of values in $2^X$
taken on by $g$ exactly once. Note that, if $g$ and $g'$ differ only at one
point, then $|L(g)-L(g')|\le2$. Construct an appropriate martingale as explained
in the Alon-Spencer book \cite[Chapter 7.4]{ASp}. 
Namely, let $V\setminus X=\{y_1,\ldots,y_m\}$
and define a sequence of auxiliary random variables $X_0,X_1,\ldots,X_m$ 
by $X_i=\condexpect{\frac12L(g)}{g(y_j)=h(y_j)\mathrm{\ for\ all\ }j\le i}$.
By Azuma's inequality (see \cite[Theorems 7.2.1 and 7.4.2]{ASp}),
for all $\lambda>0$ we have
$$
\prob{L(g)<\expect{L(g)}-2\lambda\sqrt m}<\e^{-\lambda^2/2},
$$
which is exactly what is claimed by~\refeq{eq:sieve1}.

By \refeq{eq:sieve1} we have whp that
$$
|\sieve(X)\setminus X|\ge n^{1-o(1)}.
$$
 Conditioning on $\sieve(X)$ satisfying this bound, 
we go to the next stage of generating~$\rgraph$.

\smallskip

{\it Stage 2:} reveal the edges inside $V\setminus X$.
It is enough to show that $V\setminus \sieve(X)\subset \sieve(\sieve(X)\setminus
X)$ whp. 
If the last claim is false, then there are $z,z^\prime\in V\setminus \sieve(X)$ 
having the same adjacencies to $\sieve(X)\setminus X$. This happens with
probability no more than
$$
{n\choose 2}2^{-|\sieve(X)\setminus X|}<n^22^{-n^{1-o(1)}}=o(1).
$$
The proof is complete.
\end{proof}

Theorem \ref{thm:wdrgraph} shows rather close lower and upper bounds for the logical
width and depth of a random graph $\rgraph$. Surprisingly, even this can be improved.

\begin{theorem}[Kim et al.\ \cite{KPSV}]\label{thm:precise}
For infinitely many $n$ we have whp
$$
D_2(\rgraph)\le\log n-2\log\log n+\log\log\e+6+o(1).
$$
\end{theorem}

This upper bound is at most by
$5+o(1)$ larger than the lower bound of Theorem \ref{thm:wdrgraph}. It follows that,
for infinitely many $n$, the parameters $D_i(G)$
with $i\ge 2$, $D(G)$, and
$W(G)$ are all concentrated on at most $6$ possible values
(while some extra work, see~\cite[Section~4.3]{KPSV}, gives a 5-point
concentration).

\subsubsection{Bounds for trees}

\begin{theorem}[Bohman et al.~\cite{BFL*}]\label{thm:rtree}
Let $T_n$ denote a tree on the vertex set $\{1,2,\ldots,n\}$ selected uniformly at 
random among all $n^{n-2}$ such trees. Whp we have
$W(T_n)=(1+o(1))\frac{\log n}{\log\log n}$ and $D(T_n)=(1+o(1))\frac{\log n}{\log\log n}$.
\end{theorem}

The lower bound for $W(T_n)$ immediately follows
from the following property of a random tree: whp $T_n$
has a vertex adjacent to $(1+o(1))\frac{\log n}{\log\log n}$ leaves.
Note that the upper bound for $D(T_n)$ does not follow directly from Theorem \ref{thm:dtrees}
because whp $\Delta(T_n)=(1+o(1))\frac{\log n}{\log\log n}$ (see Moon \cite{Moon}).

\subsection{An application: The convergency rate in the zero-one law}

We will write $G\models\Phi$ to say that a sentence $\Phi$ is true
on a graph $G$.
Let $p_n(\Phi)=\prob{\rgraph\models\Phi}$. The \emph{0-1 law} established
by Glebskii et al.\ \cite{GKLT} and, independently, by Fagin \cite{Fagin} says 
that, for each $\Phi$, $p_n(\Phi)$ approaches 0 or 1 as $n\to\infty$.
Denote the limit by $p(\Phi)$.

Define the \emph{convergency rate function} for the 0-1 law by
$$
R(k,n)=\max_\Phi\setdef{|p_n(\Phi)-p(\Phi)|}{D(\Phi)\le k}.
$$
Note that the maximization here can be restricted to a finite set by 
Theorem~\ref{thm:inequi2}.1.
Therefore, the standard version of the 0-1 law implies that
$R(k,n)\to 0$ as $n\to\infty$ for any fixed $k$. 
Naor, Nussboim, and Tromer~\cite{NNT}
showed that $R(\log n-2\log\log n, n)\to 0$.
Another result in~\cite{NNT} states that one can choose $p(n)=1/2+o(1)$ and
$k(n)=(2+o(1))\log n$ such that the probability that $\rpgraph$ has a
$k(n)$-clique is bounded away from 0 and 1. 
Thus for this probability $p(n)$ the 0-1 law
does not hold with respect to formulas of depth~$k(n)$.

The following theorem sharpens slightly the first part of the above
result and improves on the second part in two aspects: we do not need
to change the probability $p=1/2$ and we get an almost best possible
upper bound.

\begin{theorem}
Let $g(n)=\log n-2\log\log n+\log\log\e+c$ with constant $c$.

\begin{bfenumerate}
\item
If $c<1$, then $R(g(n),n)\to 0$ as $n\to\infty$.
\item
 If $c>6$,  then $R(g(n),n)$ does not tend to 0 as
$n\to\infty$. More strongly, for every $\gamma\in [0,1]$ there is a
sequence of formulas $\Phi_{n_1},\Phi_{n_2},\dots$ (where $n_i<n_{i+1}$) 
with $D(\Phi_{n_i})\le g(n_i)$
such that $p_{n_i}(\Phi_{n_i})\to \gamma$ as $i\to\infty$.
\end{bfenumerate}
\end{theorem}

Part 2 follows from Theorem \ref{thm:precise}. The latter implies
that (for infinitely many $n$)
actually \emph{any} property $\mathcal P$ of graphs on $n$ vertices can
be ``approximated'' by a first-order sentence of depth at most $g(n)$. 
Indeed, take the conjunction of defining
formulas over all graphs in $\mathcal P$ of order $n$ and depth at most
$g(n)$. The omitted graphs constitute negligible proportion of all
graphs by  Theorem~\ref{thm:precise}. 

We now prove Part~1. Like the
proof in~\cite{NNT} we use the extension property, but we argue 
in a slightly different way. 

\begin{proofof}{Part~1}
Let $E_k$ denote a first-order statement of quantifier depth $k$
expressing the $(k-1)$-extension property. 
Lemma~\ref{lem:alice2} 
provides us with an infinitesimal $\alpha(n)$
such that
\begin{equation}\label{eq:convergence}
1-p_n(E_k)\le\alpha(n)\text{\ \ as long as\ \ }k\le g(n).
\end{equation}

We will consider $g(x)$ on the range $x\ge2$. This function is decreasing for $x\le\e^2$
and increasing for $x\ge\e^2$. Since $g(2)<3$, for any $k_0\ge3$
there is some $n_0$ such that the conditions $g(n)\ge k_0$ and
$n\ge n_0$ are equivalent. We fix a value $k_0\ge3$ so that
$1-\alpha(n)\ge\sqrt3/2$ whenever $g(n)\ge k_0$.
Note that
\begin{equation}\label{eq:sqrt3}
p_n(E_k)\ge\frac{\sqrt3}2\text{\ \ whenever\ \ }g(n)\ge k\ge k_0.
\end{equation}

The result readily follows from the following fact.

\begin{claim}
If $k_0\le k\le g(n)$, then for every first-order statement $\Phi$ with $D(\Phi)=k$
we have 
\begin{equation}\label{eq:twoalpha}
|p_n(\Phi)-p(\Phi)|\le2\,\alpha(n).
\end{equation}
\end{claim}

We will prove first a more modest bound.

\begin{claim}
If $k_0\le k\le g(n)$, then for every first-order statement $\Phi$ with $D(\Phi)=k$
we have 
\begin{equation}\label{eq:ahalf}
|p_n(\Phi)-p(\Phi)|\le1/2.
\end{equation}
\end{claim}

\begin{subproof}
Consider a pair of integers $k$ and $n$ such that $k_0\le k\le g(n)$.
Let $\Phi$ be a first-order statement with $D(\Phi)=k$.
Without loss of generality, suppose that $p(\Phi)=0$.
If $p_N(\Phi)>1/2$ for some $N$, let $N$ denote the largest such number.
For $M=N+1$ we have $p_M(\Phi)\le1/2$.
Let $G_N$ and $G_M$ be independent random graphs with, respectively,
$N$ and $M$ vertices. Note that 
$$
\prob{D(G_N,G_M)\le k}\ge\prob{G_N\models\Phi\,\,\&\,\,G_M\not\models\Phi}>\frac14.
$$
On the other hand, Lemma \ref{lem:alice1} implies that
$$
\prob{D(G_N,G_M) > k}\ge\prob{G_N\models E_k\,\,\&\,\,G_M\models E_k}=
p_N(E_k)p_M(E_k).
$$
It follows that $p_N(E_k)p_M(E_k)<3/4$
and, therefore, $p_N(E_k)<\sqrt3/2$ or $p_M(E_k)<\sqrt3/2$.
Comparing this with \refeq{eq:sqrt3},
we conclude that $g(N)<k$ or $g(M)<k$,
which implies that $n>N$.
The desired bound \refeq{eq:ahalf} follows now by the definition of~$N$.
\end{subproof}

\setcounter{claim}{1}
\begin{subproof}
Consider a pair of integers $k$ and $n$ such that $k_0\le k\le g(n)$.
Let $\Phi$ be a first-order statement with $D(\Phi)=k$.
Let $G'$ and $G''$ be two independent copies of $\rgraph$.
By Lemma \ref{lem:alice1},
$$
\prob{D(G',G'')>k}\ge\prob{G'\models E_k\,\,\&\,\,G''\models E_k}=p_n(E_k)^2.
$$
On the other hand,
$$
\prob{D(G',G'')>k}\le\prob{G'\mathrm{\ and\ }G''\mathrm{\ are\ not\ distinguished\ by\ }
\Phi}=p_n(\Phi)^2+(1-p_n(\Phi))^2.
$$
Combining the two bounds, we obtain
$$
2\,p_n(\Phi)(1-p_n(\Phi))\le1-p_n(E_k)^2.
$$
Using Claim B, we immediately infer from here that
$$
|p_n(\Phi)-p(\Phi)|\le1-p_n(E_k)^2\le2(1-p_n(E_k)).
$$
The desired bound \refeq{eq:twoalpha} follows now from~\refeq{eq:convergence}.
\end{subproof}
\end{proofof}

\subsection{The evolution of a random graph}

We now take a dynamical view on a random graph $\rpgraph$ by letting
the edge probability $p$ vary. With $p$ varying from 0 to 1, $\rpgraph$
evolves from empty to complete. We want to trace the changes of
its logical complexity during the evolution.
Since the definability parameters do not change when we
pass to the complement of a graph, we can restrict ourselves to case $p\le1/2$.

 When $p$ is a constant, one can estimate $D(G)$ within additive
error $O(\log \log n)$.

\begin{theorem}[Kim et al.~\cite{KPSV}]\label{thm:Gnp}  
If $0<p\le1/2$ is constant, then whp
$$
\log_{1/p}n-c_1\ln\ln n-O(1)\le W(\rpgraph)
\le D(\rpgraph)\le
\log_{1/p}n + c_2\ln\ln n,
$$
where $c_1=2\ln^{-1}(1/p)$ and $c_2=(2+o(1))\of{-p\ln p-(1-p)\ln(1-p)}^{-1}-c_1$.
 \end{theorem}

\begin{sketch}
Similarly to Theorem \ref{thm:wdrgraph}, the lower bound is based
on the $k$-extension property. However, the proof of the upper bound is quite
different. 
In particular, we have hardly any control on the alternation number
in this result.
The argument is rather
complicated so we give only a brief sketch, concentrating more on its
logical rather than probabilistic component. 

Let $G=\rpgraph$ be typical and $G'\not\cong G$ be
arbitrary. Let $V=V(G)$ and $V'=V(G')$. For a sequence $X$ of
vertices, let $V_X=\{y\in V: \forall\, x\in X\ xy\in E(G)\}$ and
$G_X=G[V_X]$. Let the analogous notation (with primes) apply to
$G'$. If there is $x\in V$ such that for every $x'\in V'$ we have
$G_x\not\cong G_{x'}'$, then Spoiler selects $x$. Whatever Duplicator's
reply $x'\in V'$ is, Spoiler reduces the game to non-isomorphic graph
$G_x$ and $G_{x'}'$. We expect that $|V_x|=(p+o(1))n$ and $G_x$ is
also `typical'. Thus Spoiler used one move to reduce the order of the random
graph by a factor of $p$, which should lead to the upper bound  $D(G)\le (1+o(1))\,
\log_{1/p}n$. 

Suppose now that there are $x\in V$ and distinct $y',z'\in V'$
such that
 $G_x\cong G'_{y'}\cong G'_{z'}$.
 Spoiler selects $y'\in V'$. Assume that Duplicator replies with
$y=x$, for otherwise $G_y\not\cong G'_{y'}$ and
Spoiler proceeds as above. Now Spoiler selects $z'$; let $z\in V$
be the Duplicator's reply. We can assume that
 $G_{y,z}\cong G'_{y',z'}$,
 for otherwise Spoiler applies the inductive strategy to the
$(G_{y,z},G_{y',z'})$-game, where the order of the random graph is
reduced by factor $(1+o(1))\, p^2$. Let
$U=V_{y,z}$ and $U'=V'_{y',z'}$. A
first moment calculation shows that there is vertex $v\in
V_{y}\setminus U$ such that no vertex of $V_z\setminus U$ has
the same neighborhood in $U$ as $v$. Let Spoiler select $v$ and let $v'\in
V'_{y'}\setminus U'$ be the Duplicator's reply. Two copies
$G_{y'}'$ and $G_{z'}'$ of a `typical' graph
$G_{x}$ have a large vertex intersection. Another first moment
calculation shows that whp there is only one way to achieve this, namely
that the (unique) isomorphism $f: V_{y'}'\to V_{z'}'$ between
$G_{y'}'$ and $G_{z'}'$ is in fact the
identity on $U'$. But then $f(v')$ has the same adjacencies to
$U'$ as $v'$. Spoiler 
selects $f(v')$ and wins the game in at most one extra move.

Finally, up to a symmetry it remains to consider the case that 
there is a bijection $g:V\to V'$ such that for any
$x\in V$ we have $G_x\cong G_{g(x)}'$.

As $G\not\cong G'$, there are $y,z\in V$ such that $g$ does not
preserve the adjacency between $y$ and $z$. Spoiler selects $y$. We
can assume that Duplicator replies with $y'=g(y)$ for otherwise
Spoiler reduces the game to $G_y$. Now, Spoiler selects $z$ to which
Duplicator is forced to reply with $z'\not=g(z)$.  Let
$w=g^{-1}(z')$. Assume that
$G_{y,z}\cong G_{y',z'}$ for otherwise Spoiler applies the inductive
strategy to these graphs. But then $G_{y,z}$ is an induced subgraph of
$G_w\cong G_{z'}'$, a property that we do not expect to see in a random graph. 

In order to convert this rough idea into a rigorous proof one has to show
that whp as long as the subgraphs $G_{x_1,x_2,\dots}$ that can appear
in the game are sufficiently large, they have all required properties.
Also, one has to design Spoiler's strategy to deals
small subgraphs of $\rpgraph$ at the end of the game. All details can be found
in~\cite[Section~3]{KPSV}.
\end{sketch}

It is interesting to investigate the behavior, e.g., of
$D(\rpgraph)$ when $p=p(n)$ tends to zero. In particular, it
is open whether, for every constant $\delta\in (0,1)$ and 
$n^{-\delta}\le p(n)\le 1/2$ we have whp $D(\rpgraph)=O(\log n)$. 

Some restriction on $p(n)$ from
below is necessary here. Indeed, let $G$ be an arbitrary non-empty
graph (i.e., $G$ has at least one edge) and let $G'$ be obtained from
$G$ by adding one more isolated vertex.  It is easy to see that
$W(G,G') > d_0(G)$ and $D(G,G') > d_0(G)+1$, where $d_0(G)$ denotes
the number of isolated vertices of $G$. It follows that 
\begin{equation}\label{eq:d0}
W(G) \ge d_0(G)+1\mathrm{\ \ and\ \ }D(G) \ge d_0(G)+2. 
\end{equation}
It is well known (see,
e.g., \cite{Bol}) that 
\begin{equation}\label{eq:d0distr}
d_0(\rpgraph)=(\e^{-pn}+o(1))\, n 
\end{equation}
whp as long as $p=O(n^{-1})$. In particular, we have $W(\rpgraph)=(1-o(1))n$
whenever $p=o(n^{-1})$.

In some cases, the lower bounds \refeq{eq:d0} are sharp.  

\begin{lemma}\label{lem:comps} 
Let $c_F(G)$ denote the number of connected components in a graph $G$ 
isomorphic to a graph $F$. Suppose that a non-empty graph $G$ satisfies
\begin{equation}\label{eq:comps}
 c_F(G)+v(F)\le d_0(G)+1,\quad\mbox{for every component $F$ of $G$.}
\end{equation}
Then  $W(G)+1=D(G)=D_1(G)=d_0(G)+2$.
\end{lemma}

\begin{proof}
Let us show that $D_1(G,H)\le d_0(G)+2$ for any $H\not\cong G$.
Let $F$ be such that $c_F(H)\ne c_F(G)$. For definiteness suppose that
$c_F(H) > c_F(G)$. Spoiler marks $c_F(G)+1$ components of $H$ which are isomorphic
to $F$ by pebbling one vertex in each of them. Duplicator is forced either
to mark one of the $F$-components of $G$ twice (by pebbling two vertices, say, 
$u$ and $v$ in it) or to mark a component $F'$ of $G$ which is not isomorphic to $F$.
In the former case Spoiler wins by pebbling a path from $u$ to $v$.
In the latter case Spoiler pebbles completely the $F$-component of $H$
corresponding to $F'$. Duplicator is forced to pebble a connected part $F''$ of $F'$.
If she has not lost yet, then $F''\cong F$ and hence $F''$ is a proper subgraph
of $F'$. Spoiler wins by pebbling another vertex in $F'$ which is adjacent
to a vertex in~$F''$. Altogether at most $d_0(G)+2$ moves are made.

It remains to prove the upper bound on the width. 
The last move may require using the $(d_0(G)+2)$-th pebble.
However, for this purpose Spoiler can reuse a pebble placed earlier
in a component different from $F'$. This trick is unavailable only
if $c_F(G)=1$ and $c_F(H)=0$ or if $c_F(G)=0$ and $c_F(H)=1$. In both
cases Spoiler can win in at most $v(F)+1$ rounds (and at most one
alternation). Moreover, if this number is at least $d_0(G)+2$, then
$c_F(G)=0$ and $G$ has no component with $v(F)$ or more
vertices by~\refeq{eq:comps}. In this case, Spoiler can win in at most
$v(F)$ moves.
\end{proof}

\begin{theorem}[Kim et al.~\cite{KPSV}, Bohman et al.~\cite{BFL*}]\label{thm:invn}
If $p=c/n$ with $c=c(n)\ge 0$ being an arbitrary bounded
function of $n$,
then $D(\rpgraph)=(\e^{-c}+o(1))n$ whp.
\end{theorem}

\begin{sketch}
It is well known that, observing the evolution process in the scale $p=c/n$,
at the point $c=1$ we encounter the phase transition.
If $c<1-\epsilon$, whp all components of $\rpgraph$ have $O(\log n)$ vertices 
each; if $c>1+\epsilon$, there appears a unique exception, the so-called \emph{giant} component
with a linear number of vertices. 

One can check that, for any
$c<\alpha-\epsilon$, 
Condition \refeq{eq:comps} holds whp (even for the giant 
component if it exists), where
$\alpha=1.1918...$ is a root of some explicit equation,
see~\cite[Theorem 19]{KPSV}. Then, by Lemma \ref{lem:comps} and Equality \refeq{eq:d0distr},
$W(\rpgraph)$ and $D_1(\rpgraph)$
(and all parameters in between) are $(\e^{-c}+o(1))\, n$. 
When $c$ is larger than $\alpha+\epsilon$, then whp the
giant component of $\rpgraph$ violates~\refeq{eq:comps}: Its order exceeds the
number of isolated vertices. This case
is handled in~\cite{BFL*} as follows.

Denote the giant component of $G=\rpgraph$ by $M$. Given $H\not\cong G$, we have
to design a strategy allowing Spoiler to fast enough win the Ehrenfeucht
game on $G$ and $H$. The strategy in the proof of Lemma \ref{lem:comps}
does not work only if $c_M(H)=0$ or $c_M(H)\ge2$. We adapt it 
for these cases so that Spoiler, instead of
selecting all vertices of $M$, 
plays an optimal strategy for $M$ using at most $D(M)+\log n+1$ moves (instead of
$v(M)+1$ moves as earlier).

First, we can assume that no component of $H$ has diameter $n$ or more.
Otherwise Spoiler pebbles $u$ and $v$ at distance $n$ in $H$.
For Duplicator's responses $u'$ and $v'$ in $G$ we have either $\dist(u',v')<n$
or $\dist(u',v')=\infty$. Hence Spoiler wins in less than $\log n+1$ moves.

Second, we can assume that Duplicator always respects the connectivity
relation (two vertices are in the relation if they are connectable by a path). 
Indeed, suppose that $u$ and $v$ belong to the same connected component $F$
in one of the graphs while their counter-parts $u'$ and $v'$ are in different
components of the other graph. Then Spoiler wins in less than $\log\diam(F)+1$ moves.

Under this assumption, Spoiler easily forces that, starting from the 2nd round,
the play goes on components of $G$ and $H$, of which exactly one is isomorphic
to $M$. One of the main results of~\cite{BFL*} states that whp
\begin{equation}\label{eq:giant}
D(M)=O\of{\frac{\ln n}{\ln\ln n}},
\end{equation}
which implies that Spoiler is able to win quickly and proves the theorem.

The upper bound \refeq{eq:giant} is obtained roughly as follows. By
iteratively removing vertices of degree 1 from the giant component
$M$, one obtains the \emph{core} $C$ of $M$ (that is, $C$ is a maximum
subgraph with minimum degree at least $2$). The \emph{kernel} $K$ of
$G$ is the \emph{serial reduction} of $C$, that is, we iterate the
following to obtain $K$:
If there is a vertex $x$ of degree $2$, then we remove $x$ but add edge
$\{y,z\}$, where $y$ and $z$ are the two neighbors of $x$. The kernel may
have loops and multiple edges and has to be modeled as a \emph{colored graph}. 
The original graph $G$ can be encoded by
specifying its kernel $K$ and the structure of rooted trees that correspond to
each vertex or edge of $K$, the latter being viewed as a total
coloring of $K$. It happens that whp every vertex $x$ of $K$ can be
\emph{identified} by a small-depth formula $\Phi_x$ with one free
variable (that is $K,x\models \Phi_x$ while $K,y\not\models \Phi_x$
for every other vertex $y\in V(K)$) in the first-order language of colored
graphs. Thus one can define $K$ succinctly
by stating that for every $x\in V(K)$ there is a unique vertex
satisfying $\Phi_x$, that every vertex satisfies $\Phi_x$ for some $x\in
V(K)$, and by listing the adjacencies between vertices identified by
$\Phi_x$ and $\Phi_y$ for every $x,y\in V(K)$. The core $C$ can now be defined by specifying the
length of the path corresponding to each edge of $K$, while the giant component $M$ can be
defined by specifying the
random rooted trees hanging on the vertices of $C$ using
Theorem~\ref{thm:rtree} (which relies in part on
Theorem~\ref{thm:dtrees}). 
\end{sketch}

The bound \refeq{eq:giant} is optimal up to a constant factor.
This follows from the fact that whp the giant component $M$
has a vertex $v$ adjacent to at least $(1-\epsilon)\log n/\log\log n$
leaves. (Indeed, consider the graph $M'\not\cong M$ that is 
obtained from $M$ by attaching an extra leaf at $v$.)
We believe that the lower bound is sharp, that is, whp $D(M)=(1+o(1))
\log n/\log\log n$, but we were not able to settle this question.

Finally, we consider edge probabilities $p=n^{-\alpha}$ with
rational $\alpha\in(0,1)$. Such $p$ occur as threshold functions for (non-)appearance
of particular graphs as induced subgraphs in $\rpgraph$. What is relevant to
our subject is that such $p$ show an irregular behavior of $\rpgraph$
with respect to first-order properties. 

Since the treatment of the general case of rational $\alpha$
would require a considerable amount of technical work, the
paper~\cite{KPSV} focuses on a sample
value $\alpha=1/4$, when $D(\rpgraph)$ falls down and becomes
so small as it is essentially possible (cf.\ Section~\ref{s:best}).

\begin{theorem}[Kim et al.~\cite{KPSV}]\label{thm:arithm}
If $p=n^{-1/4}$, then whp
$$
\log^*n-\log^*\log^*n-1\le D(\rpgraph)\le D_3(\rpgraph)\le\log^*n+O(1).
$$
\end{theorem}

\begin{sketch}
The upper bound is based on the following ideas. Let the predicate
$C(x_1,x_2,x_3,x_4)$ state that these 4 distinct vertices have no
common neighbor. Its probability is $(1-p)^{n-4}=\e^{-1}+o(1)$ and
its values over different 4-tuples are rather weakly correlated. Thus,
if for a set $A$ and a vertex $v\not\in A$, we define $H_v(A)$ be the
3-uniform hypergraph on $A$ with $x_1,x_2,x_3\in A$ being a hyperedge if and
only if $C(v,x_1,x_2,x_3)$ holds, then $H_v(A)$ behaves somewhat
like a random hypergraph. As it is shown in \cite[Lemma~21]{KPSV}, one
can find 4 vertices such that their common neighborhood $A$ is
relatively large (namely, $|A|=\lfloor \ln^{0.3}n\rfloor$) and yet
there are vertices $a,m$ such that hypergraphs $H_a(A)$ and $H_m(A)$
encode in some way the multiplication and addition tables for an
initial interval of integers. Also, any integer can be succinctly defined in first-order logic with arithmetic
operations. Roughly speaking, in order to define an integer $j$, one
can write it in binary $j=b_k\dots b_1$ and specify for every $i\le k$
the $i$-th bit $b_i$; crucially, the same binary expansion trick can
be used recursively to specify the index $i$, and so on. This allows us
to identify vertices $A$ with very small depth. Next, we consider the
set $B$ of vertices of $G$ that have exactly 4 neighbors in $A$ and
are uniquely determined by this. Again, the vertices of $B$ are easy
to identify (just list the 4 neighbors in $A$). Finally, if $A$ was
chosen carefully, then each vertex $w$ of $G$ is uniquely identified by the
hypergraph $H_w(B)$. (The reason that we need an intermediate set
$B$ is that the number of possible 3-uniform hypergraphs $H_w(A)$ is at most $2^{{|A|\choose
3}}< n-|A|$, that is, too small.)
Of course, many technical difficulties arise when
one tries to realize this approach.

The lower bound in Theorem~\ref{thm:arithm} is very general. We use only the simple fact that
any particular unlabeled graph with $m$ edges is the value of $\rpgraph$, 
where $p=n^{\scriptscriptstyle-1/4}$,
with probability at most 
 $$
 n!p^m(1-p)^{{n\choose2}-m}\le n! (1-p)^{{n\choose 2}}\le \exp \left(
   -(1/2-o(1))n^{7/4} \right).
 $$
Let $F(k)$ be the number of non-isomorphic graphs definable with
depth at most $k$.  Then
$\prob{D(G)\leq k}\le F(k)\exp \ofc{ -(1/2-o(1))n^{7/4} }$. By
Theorem \ref{thm:inequi},
$F(k)\le\tower(k+2+\log^*k)$. If $k=\log^*n-\log^*\log^*n-2$, 
we have $F(k)\le 2^n$ and hence $\prob{D(G)\leq k}=o(1)$.
\end{sketch}

The above idea (arithmetization of certain vertex sets in graphs) has been
previously used by Spencer~\cite[Section~8]{Spe} to obtain
non-convergence and non-separability results on the example of
$\rpgraph$ with $p=n^{-1/3}$.

So far we have considered the evolution of the logical complexity
of a random graph in the standard logic with no counting.
We conclude this section with an extension of Theorem \ref{thm:BES}.

\begin{theorem}[Czajka and Pandurangan~\cite{CzPa}]\label{thm:CzPa} 
Let $p(n)$ be any function of $n$ such that 
$\frac{\omega(n)\,\log^4n}{n\,\log\log n}\le p(n)\le 1/2$ where $\omega(n)\to\infty$
as $n\to\infty$. Then 2 color refinements split a random graph $\rpgraph$ 
into color classes which are singletons with probability that is higher 
than $1-n^{-c}$ for each constant $c>0$ and all large enough $n$.
Consequently, $\cd2{\rpgraph}\le4$ with this probability.
\end{theorem}
Note that, in the case of $p=1/2$, this result improves the probability
bound in Part 1 of Theorem \ref{thm:BES}, while the probability
bound in Part 2 is still better.

By elaborating on the argument of Lemma \ref{lem:1colref},
we are able to supply Theorem \ref{thm:CzPa}
with the matching lower bound (that is, $\cd2{\rpgraph}\ge4$ whp)
within the range $4\sqrt{\log n/n} \le p\le 1/2$.

\section{Best-case bounds: Succinct definitions}\label{s:best}

As in the preceding sections, we consider the logical depth
of graphs with a given number of vertices $n$.
We know that the maximum value $D(G)=n+1$ is
attained by the complete and empty graphs (and only by them)
and that the typical values lie around $\log n$ (see Theorem \ref{thm:wdrgraph}).
Now we are going to look at the minimum.
We already have a good starting point: By Theorem \ref{thm:arithm},
there are graphs with
$$
D_3(G)\le\log^*n+O(1).
$$
In order to get such examples, we have to generate a random graph with
the edge probability $n^{\scriptscriptstyle-1/4}$.
In Section \ref{ss:3constr} we give three explicit constructions achieving
the same bound. In Section \ref{ss:succfunc} we introduce the \emph{succinctness function}
$q(n)=\min\setdef{D(G)}{v(G)=n}$ and give an account of what is known
about it. Section \ref{ss:zeroalt} is devoted to the question of how succinctly
we can define graphs if we are not allowed to make quantifier alternations.
In Section \ref{ss:2appl} the bounds on the succinctness function are applied to
proving separations results for logical parameters of graphs,
in particular, for $D(G)$ and~$L(G)$.

\subsection{Three constructions}\label{ss:3constr}

\subsubsection{First method: Padding}\label{sss:padding}

We describe a ``padding'' operation that was invented by Joel Spencer
(unpublished). It converts any graph $G$
to an exponentially larger graph $G^*$ with the
logical depth larger just by 1. $G^*$ includes $G$ as an induced subgraph.
In addition, for every subset $X$ of $V=V(G)$, the graph $G^*$ contains
a vertex $v_X$. Denote the set of these vertices by $V'$.
There is no edge inside $V'$ but there are some edges between
$V$ and $V'$. Specifically, $v\in V$ is adjacent to $v_X$ iff
$v\in X$. In particular, $v_\emptyset$ is isolated and $N(v_V)=V$.

Vertex $v_V$ will play a special role in our first-order definition
of $G^*$. First of all, we will say that there is a vertex $c$
(assuming $c=v_V$) whose neighborhood spans in $G^*$ a subgraph
isomorphic to $G$. This can be done by relativizing a formula
$\Phi_G$ defining $G$ to $N(c)$. That is, each universal quantification
$\A x(\Psi)$ in $\Phi_G$ has to be modified to
$$
\A_{x\in N(c)}(\Psi)\feq\A x(x\sim c\to\Psi)
$$
and each existential quantification to
$$
\E_{x\in N(c)}(\Psi)\feq\E x(x\sim c\und\Psi).
$$
Denote the relativized version of $\Phi_G$ by $\Phi_G|_{N(c)}$.
Note that relativization does not change the quantifier depth.
A sentence defining $G^*$ can now look as follows:
\begin{multline*}
\Phi_{G^*}\feq
\E c\Big(\Phi_G|_{N(c)}\und\A_{x\notin N(c)}(N(x)\subset N(c))
\und\A_{x_1\notin N(c)}\A_{x_2\notin N(c)}(N(x_1)\ne N(x_2))\\\und
\A_{x_1\notin N(c)}\A_{y\in N(x_1)}\E_{x_2\notin N(c)}(N(x_2)=N(x_1)\setminus\{y\})\Big),
\end{multline*}
where we use harmless shorthands for simple first-order expressions.

It is easy to see that
$$
D(\Phi_{G^*})=\max\ofc{D(\Phi_{G}),4}+1
$$
and that, if $\Phi_{G}$ 
is a \emph{$\E^*\A^*\E^*\A^*$-formula} (that is, every chain of nested
quantifiers is a string of this form), 
then $\Phi_{G^*}$ stays in this class as well.
Consider now a sequence of graphs $G_k$ where $G_1=P_1$, the single-vertex graph,
and $G_{k+1}=(G_k)^*$. Since $v(G^*)=v(G)+2^{v(G)}$, we have $v(G_k)\ge\tower(k-1)$.
It follows that $D_3(G_k)\le\log^*v(G_k)+3$.

\subsubsection{Second method: Unite and conquer}\label{sss:unite}

Suppose that we have a set $C$ of $n$-vertex graphs, each of 
logical depth at most $d$. 
Our goal is to construct a much larger set $C^*$
of graphs with a much larger number of vertices $n^*$ and logical depth
bounded by $d+3$. An additional technical condition is that all the graphs
have diameter 2. We know from Theorem \ref{thm:wdrgraph} that almost all graphs 
on $n$ vertices have logical depth less than $\log n$ and it is well known
that they have diameter 2. Choosing a sufficiently large $n$, we can 
start with $C$ being the class of all such graphs. Since almost all graphs
are asymmetric, we have $|C|=(1-o(1))2^{n\choose2}$. 
Just for the notational simplicity, we prefer that $|C|$ is even.

For each $S\subset C$ such that $|S|=|C|/2$, the set $C^*$ contains graph
$$
G_S=\compl{\bigsqcup_{G\in S}G},
$$
that is, we take the vertex disjoint union of all graphs in $S$ and
complement it. For convenience, we bound the logical depth of the
complement $\compl{G_S}=\bigsqcup_{G\in C}G$ rather than that of
$G$. Given an arbitrary $H\not\cong\compl{G_S}$, we analyze the
Ehrenfeucht game on the two graphs.

If $H$ has a connected component of diameter at least 3, Spoiler pebbles
vertices $u$ and $v$ in $H$ at the distance exactly 3 from one another.
For Duplicator's responses $u'$ and $v'$ in $\compl{G_S}$,
either $\dist(u',v')\le2$ or $\dist(u',v')=\infty$. In any case, Spoiler wins within
the next 2 moves. Suppose from now on that all components of $H$ have diameter
at most 2. This condition allows us to assume that Duplicator respects
the connectivity relation for otherwise Spoiler wins with one extra
move (which will be added to the total count of rounds).

If one of the graphs, $\compl{G_S}$ or $H$, has a connected component $A$ non-isomorphic to
any component of the other graph, Spoiler pebbles a vertex in $A$. Let $B$
be the component of the other graph where Duplicator responds. Starting from the
second round, Spoiler plays the Ehrenfeucht game on non-isomorphic graphs $A$ and $B$
and wins in at most $d$ moves. 

If such a component does not exist, 
$\compl{G_S}$ must have a component $A$
with at least two isomorphic copies in $H$. Then in the first two rounds Spoiler 
pebbles vertices in these two. Duplicator is forced at least once to respond
in a component $B$ of $\compl{G_S}$ non-isomorphic to $A$, which is an already
familiar configuration.

Thus, $D(G_S)$ can be at most 3 larger than the maximum logical depth
of graphs in $C$. At the same time $G_S$ has the much larger number of vertices,
namely $n^*=n|C|/2$. It follows that $D(G)<\log\log v(G)$ for any $G$ in $C^*$.

Note that any graph in $C^*$ is the complement of a disconnected graph
and hence has diameter 2. This allows us to iterate the construction.
Say, for any $G\in(C^*)^*$ we get $D(G)<\log\log\log v(G)$ and so on.
If we fix the initial class $C$, 
the iteration procedure gives us
graphs with $D(G)<3\log^* v(G)+O(1)$. This bound is worse than in the
preceding section but the extra factor of 3 can be eliminated if 
Spoiler plays more smartly (see~\cite{PSV2}).

\subsubsection{Third method: Asymmetric trees}\label{sss:asymtrees}

The two previous examples were artificially constructed
with the aim to ensure low quantifier depth. Now we present
a natural class of graphs admitting succinct definability.

The {\em radius\/} of a graph $G$ is
defined by $r(G)=\min_{v\in V(G)}e(v)$, where $e(v)$ denotes the eccentricity
of a vertex $v$.
A vertex $v$ is {\em central\/} if $e(v)=r(G)$.
Any tree has either one or two central vertices (see, e.g., \cite[Chapter 4.2]{Ore}).

\begin{lemma}\label{lem:treeDr}
Let $T$ be an asymmetric tree with $r(T)\ge6$.
Then $D(T)\le r(T)+2$.
\end{lemma}

\begin{proof}
We will design a strategy for Spoiler in the Ehrenfeucht
game on $T$ and a non-isomorphic graph $T'$. 
The reader that took the effort to
reconstruct the proof of Theorem \ref{thm:wtrees} will now
definitely benefit.

We can assume that $T$
and $T'$ have equal diameters 
(in particular, $T'$ is connected)
for else Spoiler wins in less than $\log r(T)+4$ moves by Lemma \ref{lem:distance}.
If $T'$ is a non-tree, let Spoiler pebble a vertex $v'$ on
a cycle in $T'$.
By this move Spoiler forces the game on $T\setminus v$ and $T'\setminus v'$,
where $v$ is Duplicator's response in $T$.
If $v$ is a leaf, Spoiler wins in two moves. Otherwise $T\setminus v$
is disconnected, while $\diam(T'\setminus v')\le3\,\diam(T')\le6\,r(T)$.
Lemma \ref{lem:distance} applies again and Spoiler wins in less than $\log r(T)+6$ moves.
Assume, therefore, that $T'$ is a tree too.

Call a tree \emph{diverging} if every vertex $w$ splits it into pairwise
non-isomorphic branches, where each branch is considered rooted
at the respective neighbor of $w$ (an isomorphism of rooted trees has
to match their roots). Any asymmetric tree is obviously
diverging. On the other hand, if a tree is diverging, it is either
asymmetric or has a single nontrivial automorphism and the latter
transposes two central vertices.

Suppose that $T'$ is diverging.
In the first round Spoiler pebbles a central vertex $v$ of $T$
and Duplicator responds with a vertex $v'$ in $T'$.
As it is easily seen, at least one of $T\setminus v$ and $T'\setminus v'$ has a branch $B$
non-isomorphic to any branch in the other tree. Spoiler restricts
further play to $B$ by pebbling its root. Continuing in this fashion,
that is, each time finding a matchless subbranch, Spoiler forces
pebbling two paths in $T$ and $T'$ emanating from $v$ and $v'$ respectively.
Spoiler wins at latest when the path in $T$ reaches a leaf.

So suppose that $T'$ is not diverging. Let $v'$ be
a central vertex of $T'$ and $u'$ be a vertex at the maximum possible distance
from $v'$ with the property that $T'\setminus u'$ has two isomorphic branches
$B'$ and $B''$. Spoiler pebbles the path from $v'$ to $u'$ and the two
neighbors of $u'$ in $B'$ and $B''$. From this point Spoiler can play as before
because $B'$ and $B''$ are diverging and only one of them can be isomorphic
to the corresponding branch pebbled by Duplicator in $T$.
\end{proof}

Lemma \ref{lem:treeDr} shows that asymmetric trees are definable with
quantifier depth not much larger than their radius.
On the other hand, asymmetric trees can grow in breadth,
having a huge number of vertices. 
More precisely, there are asymmetric trees with $v(T)\ge\tower(r(T)-1)$.
Indeed, let $r_k$ denote
the number of asymmetric rooted trees of height at most $k$.
A simple recurrence 
$$
r_0=1,\quad
r_k=2^{r_{k-1}}
$$
shows that $r_k=\tower(k)$.
Let $k\ge3$ and $T_k$ be the (unrooted) tree of radius $k$ with a single central
vertex $c$ such that the set of branches growing from $c$
consists of all $r_{k-1}$ pairwise non-isomorphic asymmetric rooted trees
of height less than $k$. (The reader will now surely recognize
another instance of the unite-and-conquer method!)
Since $T_k$ has even diameter, the central vertex $c$ is fixed under all automorphisms.
It easily follows that $T_k$ is asymmetric.
This graph will be referred to as the \emph{universal asymmetric tree of radius $k$}.
Note that $v(T_k)\ge r_{k-1}+1>\tower(k-1)$.
Combining it with Lemma \ref{lem:treeDr}, we obtain
$D(T_k)\le k+2\le\log^*v(T_k)+2$.

With a little extra work, trees with low logical depth can be constructed
on any given number of vertices. It turns out that the log-star bound
is essentially the best what can be achieved for trees.

\begin{theorem}[Pikhurko, Spencer, and
  Verbitsky~\cite{PSV}]\label{thm:BestTree}
For every $n$ there is a tree $T$ on $n$ vertices with $D(T)\le\log^*n+4$.
On the other hand, for all trees $T$ on $n$ vertices we have
$D(T)\ge\log^*n-\log^*\log^*n-4$.
\end{theorem}

We will see in the next section that the lower bound of Theorem \ref{thm:BestTree} 
cannot be extended to the class of all graphs. 

Universal asymmetric trees have been proved to be a useful technical tool
in complexity theory and finite model theory since a long time, see the
references in Dawar et al.\ \cite{DGKS}. Lemma 3.4(e) in the latter paper
readily implies a succinctness result for the logical \emph{length}.

\begin{theorem}[Dawar et al.~\cite{DGKS}]\label{thm:UnivTree}
For the universal asymmetric tree of radius $k$ we have
$L(T_k)=O((\log^*v(T_k))^4)$.
\end{theorem}
The theorem shows that, for infinitely many $n$, there is a tree $T$ on $n$
vertices with $L(T)=O((\log^*n)^4)$. Unlike Theorem \ref{thm:BestTree},
this result cannot be extended to all $n$ because there are infinitely many $n$
such that all graphs on $n$ vertices have logical length $\Omega\of{\frac{\log n}{\log\log n}}$
(see \refeq{eq:gnlow} in the proof of Theorem~\ref{thm:tight}).

\subsection{The succinctness function}\label{ss:succfunc}

Define the \emph{succinctness function} by
$$
q(n)=\min\setdef{D(G)}{v(G)=n}.
$$
Since only finitely many graphs are definable with a fixed
quantifier depth (see Theorem \ref{thm:inequi}), we have
$q(n)\to\infty$ as $n\to\infty$. The examples collected in Section \ref{ss:3constr}
show that $q(n)$ increases rather slowly. Let $q_a(n)$ denote the version
of $q(n)$ for definitions with at most $a$ quantifier alternations.
The padding construction from Section \ref{sss:padding} gives us
\begin{equation}\label{eq:q3logstar}
q_3(n)\le\log^*n+3
\end{equation}
for infinitely many $n$ and, by Theorem \ref{thm:arithm},
this bound holds actually for all $n$, perhaps with a worst
additive constant.

Is the log-star bound best possible? The answer is surprising enough:
in some strong sense it is but, at the same time, it is very far from
being tight. First, let us elaborate on the latter claim.

A {\em prenex formula\/} is a formula with all its quantifiers being in front.
In this case there is a single sequence of nested quantifiers and
the quantifier rank is just the number of quantifiers
occurring in a formula.
The superscript {\it prenex} will mean that we allow defining sentences
only in prenex form. Thus, $q_a^\prenex(n)$ is equal to the minimum quantifier
depth of a prenex formula with at most $a$ quantifier alternations
that defines a graph on $n$ vertices. We obviously have 
$D(G)\le D_a(G)\le D_a^\prenex(G)$.
Recall that $L_a(G)$ denotes the minimum length of a
sentence defining $G$ with at most $a$ quantifier alternations.
Since a quantifier-free formulas with $k$ variables
is equivalent to a disjunctive normal form over $2{k\choose2}$ relations
between the variables, we obtain also relation
\begin{equation}\label{eq:LDprenex}
L_a(G)=O(h(D_a^\prenex(G)))\mathrm{\ \ where\ \ }h(k)=k^22^{k^2}.
\end{equation}

Recall that a {\em total recursive function\/} is 
an everywhere defined recursive function.

\begin{theorem}[Pikhurko, Spencer, and Verbitsky~\cite{PSV}]\label{thm:superrec}
There is no total recursive function $f$ such that
$f(q_3^\prenex(n))\ge n$ for all~$n$.
\end{theorem}
The theorem implies a superrecursive gap between $v(G)$ and $D_3(G)$
or even $L_3(G)$. In particular, the values of $q_3(n)$ are infinitely often
inconceivably smaller even than the values of $\log^*n$.
More generally, if a total recursive function $l(n)$
is monotone nondecreasing and tends to infinity, then
\begin{equation}\label{eq:intro0}
q(n)<l(n)\mbox{\ \ for infinitely many\ \ }n,
\end{equation} 
which actually means that the succinctness function admits 
no reasonable lower bound.

The proof of Theorem \ref{thm:superrec} is based on simulation of 
a Turing machine $M$ by a prenex formula $\Phi_M$ in which a computation of $M$
determines a graph satisfying $\Phi_M$ and vice versa.
Such techniques were developed in the classical research on Hilbert's
{\em Entscheidungsproblem\/} by Turing, Trakhtenbrot, B\"uchi and
other researchers (see \cite{BGG} for survey and references).
An important feature of our simulation is that it works if we restrict
the class of structures to graphs. 
As a by-product, we obtain another proof of Lavrov's version
of the Trakhtenbrot theorem \cite{Lav} (see also \cite[Theorem 3.3.3]{ELTT}) saying
that the first-order theory of finite graphs is undecidable.
The proof actually shows the undecidability of the
$\forall^*\exists^p\forall^s\exists^t$-fragment of this theory
for some $p$, $s$, and $t$.

We now have to explain why bound \refeq{eq:q3logstar}, though not sharp,
is best possible in some sense.
Let us define the {\em smoothed succinctness function\/} $q^*(n)$
to be the least monotone nondecreasing
integer function bounding $q(n)$ from above, that is,
\begin{equation}\label{eq:smoothed}
q^*(n)=\max_{m\le n} q(m).
\end{equation}
The following theorem shows that $q^*(n)=(1+o(1))\log^*n$ and, therefore,
the log-star function is a nearly optimal \emph{monotone} upper bound
for the succinctness function~$q(n)$.

\begin{theorem}[Pikhurko, Spencer, and Verbitsky~\cite{PSV}]\label{thm:ssf}
 $$
 \log^*n-\log^*\log^*n-2\le q^*(n)\le\log^*n+4. 
 $$
\end{theorem}
Though the lower bound contains a nonconstant lower order term,
it can hardly be distinguished from a constant: for example,
$\log^*\log^*n=3$ for $n=10^{80}$, which is a rough estimate of
the number of elementary particles in the observable universe.
\begin{proof}
Theorem \ref{thm:BestTree} implies that $q(n)\le\log^*n+4$ for all $n$.
Since this bound is monotone, it is a bound on $q^*(n)$ as well.
The lower bound for $q^*(n)$ can be derived from Theorem \ref{thm:inequi}.
According to it, at most $\tower(k+\log^*k+2)$
graphs are definable with quantifier depth $k$.
Given $n>\tower(3)$, let $k$ be such that
$\tower(k+2+\log^*k)<n\le\tower(k+3+\log^*(k+1))$.
It follows that
$k>\log^*n-\log^*\log^*n-4$.
By the Pigeonhole Principle, there will be some
$m\le n$ for which no graph of order precisely $m$ is defined
with quantifier depth at most $k$. We conclude that $q^*(n)\ge q(m)>k$
and hence $q^*(n)\ge \log^*n-\log^*\log^*n-2$.
\end{proof}

We defined $q^*(n)$ to be the ``closest'' to $q(n)$
monotone function. Notice that $q(n)$ itself lacks the monotonicity,
deviating from $q^*(n)$ infinitely often (set $l(n)$ to be the lower
bound in Theorem \ref{thm:ssf} and apply~\refeq{eq:intro0}).

\subsection{Definitions with no quantifier alternation}\label{ss:zeroalt}

It is interesting to observe how the succinctness function changes
when we put restrictions on the logic.
Note that all what we have stated about the succinctness function
for first-order logic actually holds true for its fragment
with 3 quantifier alternations.
 Now we consider
the first-order logic with no
quantifier alternation, consisting of purely existential and purely universal formulas
and their monotone Boolean combinations (of course, all negations
are supposed to stay in front of relation symbols).
It is easy to see that any sentence with no quantifier alternation is equivalent to
a sentence in the \emph{Bernays-Sch\"onfinkel class}. The latter consists of prenex formulas
in which the existential quantifiers all precede the universal quantifiers, as in
\begin{equation}\label{eq:BS}
\Phi\feq\E x_1\ldots\E x_k\A y_1\ldots \A y_l \Psi(\bar x,\bar y),
\end{equation}
where $\Psi$ is quantifier-free.
This fragment of first-order logic is provably weak.

To substantiate this claim, consider the \emph{finite satisfiability problem}:
Given a first-order sentence $\Phi$ about graphs, one has to
decide whether or not there is a finite graph satisfying $\Phi$.
More generally, let $\mathrm{Spectrum}(\Phi)$ consist of all those $n$ such that
there is a graph on $n$ vertices satisfying $\Phi$. Thus, the problem
is to decide whether $\mathrm{Spectrum}(\Phi)$ is nonempty.

Lavrov \cite{Lav} proved
that this problem is unsolvable even for sentences without equality
(for directed graphs this is a classical result on Hilbert's {\em Entscheidungsproblem},
known as the Trakhtenbrot-Vaught theorem, see \cite{BGG}).
However, if we consider only sentences in the Bernays-Sch\"onfinkel class,
the finite satisfiability problem becomes decidable. 
This directly follows from the following simple observation
showing that a nonempty spectrum always contains a certain small number.

\begin{lemma}\label{lem:BS}
Suppose that a first-order sentence $\Phi$ is of the form \refeq{eq:BS}.
If $\Phi$ is satisfiable, then $\mathrm{Spectrum}(\Phi)$ contains $k$
or a smaller number.
\end{lemma}

\begin{proof}
Assume that $\Phi$ is true on a graph $G$ with more than $k$ vertices
and let $U\subset V(G)$ be the set of vertices $x_1,\ldots,x_k$ whose existence 
is claimed by $\Phi$. Note that the induced subgraph $G[U]$
satisfies $\Phi$ as well.
\end{proof}

The solvability of the finite satisfiability problem for the Bernays-Sch\"onfinkel class
was observed by Ramsey in \cite{Ram}. Ramsey showed that the spectrum of a Bernays-Sch\"onfinkel
formula can be completely determined. This follows from the following result where
his famous combinatorial theorem appeared as a technical tool.
Recall that a set is \emph{cofinite} if it has finite complement.

\begin{theorem}[Ramsey \cite{Ram}]\label{thm:ramsey}
Any sentence about graphs $\Phi$ in the Bernays-Sch\"onfinkel class 
has either finite or cofinite
spectrum. More specifically, if $\Phi$ is of the form \refeq{eq:BS}, then
either $\mathrm{Spectrum}(\Phi)$ contains no number equal to or greater than
$2^k4^l$ or it contains all numbers starting from $k+l$.
\end{theorem}

\begin{proof}
Assume that $\Phi$ is true on a graph $G$ with at least $2^k4^l$ vertices
and let $U\subset V(G)$ consist of vertices $x_1,\ldots,x_k$ whose existence is claimed
by $\Phi$. Recall that \emph{Ramsey number} $R(l)$ is equal to the minimum $R$
such that every graph with $R$ or more vertices contains a homogeneous
set of $l$ vertices. As it is well known, $R(l)<4^l$. By the Pigeonhole
Principle, 
$V(G)\setminus U$ contains a subset $W$ of $R(l)$ vertices with the same
neighborhood within $U$. Let $X$ be a homogeneous set of $l$
vertices in $G[W]$. 
Note that $G[U\cup X]$ satisfies $\Phi$ and that $X$ is a set of $l$ twins
in this graph. Cloning the twins, we can obtain a graph that satisfies 
$\Phi$ and has any number of vertices larger than $k+l$.
\end{proof}

After this small historical excursion, let us turn back to the definability
with no quantifier alternation. First of all, note that
even without quantifier alternation all graphs remain definable 
(see \refeq{eq:def}) and, hence, the parameter $D_0(G)$ is well defined.

\begin{theorem}[Pikhurko, Spencer, and
  Verbitsky~\cite{PSV}]\label{thm:D0comp}
$D_0(G)$ is a computable parameter of a graph.
\end{theorem}

\begin{proof}
Given $m\ge0$, one can algorithmically construct a finite set $U_m$
consisting of 0-alternating sentences of quantifier depth
$m$ so that every 0-alternating sentence of quantifier depth $m$
has an equivalent in $U_m$. To decide if $D_0(G)\le m$, for each
sentence $\Upsilon\in U_m$ satisfied by $G$ we have to check
if $\Upsilon$ can be satisfied by another graph $G'$. We first reduce $\Upsilon$
to an equivalent statement $\Psi$ in the Bernays-Sch\"onfinkel class.
Suppose that $\Psi$ has $k$ existential quantifiers. It suffices to
test all $G'\not\cong G$ with at most $k+1$ vertices. Indeed, if
$\Phi$ is true on a graph with more than $k+1$ vertices then,
by the argument used to prove Lemma \ref{lem:BS}, $\Phi$ is as well true on 
its induced subgraphs with $k+1$ and $k$ vertices (one of which
is not isomorphic to $G$).
\end{proof}

We cannot prove anything similar for $D(G)$ or even $D_1(G)$.
The proof of Theorem \ref{thm:D0comp} is essentially based
on the decidability of whether or not a 0-alternating sentence
is defining for some graph. However, in general this problem 
is undecidable (see~\cite{PSV}).

For the logic with no quantifier alternation, the succinctness
function has much more regular behavior.

\begin{theorem}[Spencer, Pikhurko, and Verbitsky~\cite{PSV2}]\label{thm:q0}
 $$ 
 \log^*n-\log^*\log^*n-2\le q_0(n)\le \log^*n+22.
 $$
\end{theorem}

The lower bound has to be contrasted to Theorem \ref{thm:superrec}.
It gives us a kind of a quantitative confirmation of the fact that
the 0-alternation fragment of first-order logic is strictly less powerful. 
The upper bound improves upon the alternation number in \refeq{eq:q3logstar}
attaining the optimum. The proof of this bound is based on the 
unite-and-conquer construction in Section \ref{sss:unite},
where more subtle analysis is needed in order to achieve the zero alternation
number. All the details can be found in~\cite{PSV2}.

\begin{proofof}{Theorem \ref{thm:q0} (lower bound)}
Given $n$, denote $k=q_0(n)$ and fix a graph $G$ on $n$ vertices such that
$D_0(G)=k$. The same relation between $L_a(G)$
and $D_a(G)$ as in Theorem \ref{thm:dvsl} is proved in \cite{PSV2}. 
By this result,
$G$ is definable by a 0-alternating sentence
$\Upsilon$ of length less than $\tower(k+\log^*k+2)$. Convert
$\Upsilon$ to an equivalent sentence $\Phi$ in the Bernays-Sch\"onfinkel class
and note that $D(\Phi)\le L(\Upsilon)$. 
By Lemma \ref{lem:BS}, $\Phi$ must be true on some graph with
at most $D(\Phi)$ vertices. Since $\Phi$ is true only on $G$, we have 
 $$
 n\le D(\Phi)\le L(\Upsilon)< \tower(k+\log^*k+2).
 $$
 This implies that 
 \begin{equation}\label{n1}
 \log^* n\le k+\log^*k+2.
 \end{equation}
 Suppose on the contrary to our claim that 
$k\le\log^*n-\log^*\log^*n-3$. Then $\log^*k\le \log^*\log^*n$ and \refeq{n1} implies
that 
 $$
 \log^* n \le (\log^*n-\log^*\log^*n-3) + \log^*\log^*n+2,
 $$
 which is a contradiction, proving the claimed bound.
\end{proofof}

Using the lower bound of Theorem \ref{thm:q0} and the absence 
of any recursive linkage between $q_3(n)$ and $n$, we are able to
show a superrecursive gap between two parameters in the logical
depth hierarchy
$$
D(G)\le D_3(G)\le D_2(G)\le D_1(G)\le D_0(G).
$$

\begin{theorem}[Pikhurko, Spencer, and
  Verbitsky~\cite{PSV}]\label{thm:dvsd0}
There is no total recursive function $f$ such that
$D_0(G)\le f(D_3(G))$ for all graphs~$G$.
\end{theorem}

\begin{proof}
Assume that such an $f$ exists.
Let $G_n$ be a graph for which $D_3(G_n)=q_3(n)$. Then
$$
f(q_3(n))=f(D_3(G_n))\ge D_0(G_n)\ge q_0(n)\ge \log^*n-\log^*\log^*n-2.
$$
This implies that $\tower(2f(q_3(n)))\ge n$, contradictory to Theorem~\ref{thm:superrec}.
\end{proof}

We have seen weighty evidences that the 0-alternating sentences are strictly
less expressive than the sentences of the same quantifier depth with
quantifier alternations. It is quite surprising that, nevertheless, sometimes
we can prove for $D_0(G)$ upper bounds which are just a little
worse than the best known bounds for $D(G)$. The following results should be
compared with Theorems \ref{thm:dtrees}, \ref{thm:PVV}, and~\ref{thm:wdrgraph}.

\begin{theorem}\label{thm:allzero}\mbox{}
\begin{bfenumerate}
\item
{\rm(Bohman et al.~\cite{BFL*})}
Let $D_0(n,d)$ denote the maximum of $D_0(T)$ over all
trees with $n$ vertices and maximum degree at most $d=d(n)$.
If both $d$ and $\log n/\log d$ tend to infinity, then
$
D_0(n,d)\le (1+o(1)) \frac{d \log n}{\log d}
$.
\item
{\rm(Pikhurko, Veith, and Verbitsky~\cite{PVV})}
$D_0(G,H)\le\frac{n+5}2$ for all non-isomorphic graphs
$G$ and $H$ with the same number of vertices~$n$.
\item
{\rm(Kim et al.~\cite{KPSV})}
$D_0(\rgraph)\le(2+o(1))\log n$ with high probability.
\end{bfenumerate}
\end{theorem}

We conclude this subsection with a demonstration of somewhat surprising strength
of the Bernays-Sch\"onfinkel class. We say that a sentence $\Phi$ \emph{identifies}
a graph $G$ if it distinguishes $G$ from any non-isomorphic graph of the same order.
Let $\bs G$ denote the minimum quantifier depth of $\Phi$ in the Bernays-Sch\"onfinkel 
class identifying $G$. We already discussed the identification problem in 
Section \ref{sss:identif}. Note, however, a striking difference. While
in Section \ref{sss:identif} we could make the conjunction of all sentences $\Phi_H$
distinguishing $G$ from another graph $H$ of the same order, now we have to
distinguish $G$ from all such $H$ by a single prenex sentence!

\begin{theorem}[Pikhurko and Verbitsky~\cite{PVe}]\mbox{}
\begin{bfenumerate}
\item
For any graph $G$ of order $n$,
we have $\bs G\le\frac34\,n+\frac32$.
\item
With high probability we have $\bs{\rgraph}\le (2+o(1))\log n$.
Moreover, the latter bound holds true even if the number of universal
quantifiers in an identifying formula is restricted to~2.
\end{bfenumerate}
\end{theorem}

\subsection{Applications: Inevitability of the tower function}\label{ss:2appl}
 
Succinctly definable graphs can be used to show that the tower function
is sometimes unavoidable in relations between logical parameters of graphs.
We first observe that the relationship between the logical depth
and the logical length in Theorem \ref{thm:dvsl} is ``nearly'' tight.

\begin{theorem}[Pikhurko, Spencer, and Verbitsky~\cite{PSV}]%
\mbox{}\hspace{-2.5mm}\footnote{%
In \cite{PSV} we stated a better bound $L(G)\ge\tower(D(G)-6)-O(1)$,
which was proved for the variant of $L(G)$ where variable $x_i$
contributes $\log i$, rather than just 1, to the formula length.}\label{thm:tight}
There are infini\-te\-ly many pairwise non-isomorphic graphs $G$ with
$L(G)\ge\tower(D(G)-7)$.
\end{theorem}

\begin{proof}
The proof is given by a simple counting argument.
A first-order sentence $\Phi$ defining a graph $G$
determines a natural binary encoding of $G$ (up to isomorphism)
of length $O(L(\Phi)\log L(\Phi))$.
It follows that at most $m=2^{O(k\log k)}$ graphs can have
logical length less than $k$. By the Pigeonhole Principle,
there is $n\le m+1$ such that $L(G)\ge k$ for all $G$ on $n$ vertices.
For all these graphs we have
\begin{equation}\label{eq:gnlow}
L(G)=\Omega\of{\frac{\log n}{\log\log n}},
\end{equation}
which exceeds $\log\log n$ if $k$ is chosen sufficiently large.
By Theorem~\ref{thm:BestTree}, there is a graph $G_n$ on $n$ vertices with 
\begin{equation}\label{eq:gnup}
D(G_n)<\log^*n+5.
\end{equation}
Combining \refeq{eq:gnup} and \refeq{eq:gnlow}, we obtain the desired
separation of $L(G_n)$ from $D(G_n)$.
Increasing the parameter $k$, we can have infinitely many such examples.
\end{proof}

One of the consequences of Theorem \ref{thm:superrec} is that prenex formulas are sometimes
unexpectedly efficient in defining a graph. We are now able to show that,
nevertheless, they generally cannot be competitive against defining formulas
with no restriction on structure. More specifically, we have simple relations
\begin{equation}\label{eq:DDLLprenex}
D(G)\le D^\prenex(G)<L(G)\le L^\prenex(G).
\end{equation}
Combining the second inequality with Theorem \ref{thm:dvsl},
we obtain
$$
D^\prenex(G)<\tower(D(G)+\log^* D(G)+2)
$$
and we can now see that this relationship between $D^\prenex(G)$ and $D(G)$
is not so far from being optimal.

\begin{corollary}\label{cor:pren}
There are infinitely many pairwise non-isomorphic graphs $G$ with
$D^\prenex(G)\ge\tower(D(G)-8)$.
\end{corollary} 
The proof of Theorem \ref{thm:tight} gives us actually a better bound, 
though somewhat cumbersome, namely $L(G)\ge T/(c\log T)$ with $T=\tower(D(G)-6)$
and $c$ a constant. Corollary \ref{cor:pren} follows from here simply by
noticing that parameters $D^\prenex(G)$ and $L(G)$ are exponentially close.
The latter fact follows from \refeq{eq:DDLLprenex} 
and a version of \refeq{eq:LDprenex},
namely
$$
L(G)=O(h(D^\prenex(G)))\mathrm{\ \ where\ \ }h(x)=x^22^{x^2}.
$$

In conclusion we note that the tower function is essential also 
in the upper bound for the number of graphs definable with
quantifier depth $k$ given by Theorem~\ref{thm:inequi}.

\begin{corollary}
There are at least $(1-o(1))\tower(k-2)$ first-order sentences of quantifier
depth $k$ defining pairwise non-isomorphic graphs and, hence,
being pairwise inequivalent.
\end{corollary}

\begin{proof}
In Section \ref{sss:asymtrees} we noticed that there are exactly
$r_h=\tower(h)$ asymmetric rooted trees of height at most $h$.
Basically this follows from the fact that such a tree is completely
characterized by the set of its branches from the root, each being
an asymmetric rooted tree of height at most $h-1$ (the root is not a part
of any branch). Thus, $r_h-r_{h-1}$ asymmetric rooted trees have height exactly $h$.
Note that $(r_{h-1}-r_{h-2})r_{h-1}$ of them have exactly one branch
of height $h-1$. Therefore, there are at least 
$r_h-r_{h-1}-(r_{h-1}-r_{h-2})r_{h-1}=(1-o(1))\tower(h)$ asymmetric rooted trees
whose underlying trees (with roots dismissed) have diameter $2h$ and,
hence, are asymmetric too. By Lemma~\ref{lem:treeDr}, each of these trees 
is definable with quantifier depth~$h+2$.
\end{proof}

A lower bound of $\tower(k-2)$ for the number of pairwise inequivalent
sentences of quantifier depth $k$ is shown by Spencer \cite[Theorem 2.2.2]{Spe}.

\section{Open problems}\label{s:open}

Many questions remain open, some of which are included in the main
text of the survey alongside the known related results. For reader's
convenience we collect a few open problems here that we consider most
interesting.

Tomasz \L uczak (Conference on Random Structures and Algorithms, Pozna\'n, 2003) 
asked if $D(G)$, or $W(G)$, is a computable function of the input graph~$G$.

While the factor of $1/2$ in Theorem~\ref{thm:PVV} is best possible, we
do not know if it can be improved for logic with
counting. Surprisingly, we could not resolve even the following
question. Is there $\epsilon >0$ such that for every graph $G$ of
sufficiently large order $n$ we have $\cw G\le(\frac12-\epsilon)n$?

Recall that no sublinear bound is generally possible here because
Cai, F\"urer, and Immerman \cite{CFI} constructed graphs
with linear width in the counting logic; see Theorem \ref{thm:CFI}.
Automorphisms of these graphs play an essential role in establishing
this lower bound. It would be very interesting to estimate $\cw G$
from above for asymmetric $G$. 
Again, we have only the bound $\cd{}G\le(n+3)/2$
as a straightforward corollary of Theorem \ref{thm:PVV},
where no restriction on the automorphism group is supposed.

Another research direction, with applications to the graph isomorphism problem,
is identification of natural classes of graphs with $\cw G$ bounded by a constant;
see Sections \ref{sss:exclminor} and \ref{sss:others}. 
Such a bound is known for interval graphs \cite{EPT00,Lau},
and it is interesting if it can be extended to the class
of circular-arc graphs.
The approach suggested in \cite{Lau} is based on the fact that
any maximal clique in an interval graph is definable as the common neighborhood
of some two vertices. This prevents any straightforward extension to 
circular-arc graphs, where the number of maximal cliques can be exponential.
The (un)boundedness of $\cw G$ is an interesting open question also for
disk graphs, yet another extension of the class of interval graphs (Martin Grohe, 2010).

A result of Dawar, Lindell, and Weinstein \cite{DLW} (see also Theorem \ref{thm:Dk})
implies an upper bound for $\cd{}G$ in terms of $\cw G$ and the order $n$ of $G$,
where $\cw G$ disappointedly occurs at the exponent. Can this bound be improved?
At the moment we cannot even exclude that $\cd{}G=O(\cw G\log n)$.
If the latter bound was true for $\cd kG$ with $k=O(\cw G)$, this would have
important consequences for isomorphism testing by Theorem~\ref{thm:GVe}.

Where do we need the power of counting quantifiers?
To keep far away from the trivial example of a complete or empty graph,
suppose that a graph $G$ is asymmetric. Is it true or not that
$W(G)=O(\cw G\log n)$? A random graph shows that this bound would be
best possible.

We are still far from having a complete evolutionary picture of
the logical complexity for a random graph.
Let $\delta\in(0,1)$ be fixed and $p$ be an arbitrary function of
$n$ with $n^{-\delta}\le p \le \frac12$. Is it true that whp
$D(\rpgraph)=O(\log n)$?

The local behavior of the succinctness function $q(n)$,
that was defined in Section~\ref{ss:succfunc}, is unclear. While it is trivial that
$q(n+1)\le q(n)+1$, we do not know, for example, if $q(n+1)\ge
q(n)-C$ for some constant $C$ and all $n$.

In accordance with our notation system, let $q^k(n)$ denote the
succinctness function for the $k$-variable logic. By slightly
modifying the proof of Lemma \ref{lem:treeDr}, one can show that
$q^3(n)\le(1+o(1))\log^*n$ for all $n$. Since the satisfiability
problem for the 3-variable logic is undecidable (see, e.g., \cite{Gro:survey}),
it is not excluded that an analog of Theorem \ref{thm:superrec} 
can be established for~$q^3(n)$.

Given a fixed $k$, how far apart from one another can the values of
$D(G)$ and $D^k(G)<\infty$ be?

Theorem~\ref{thm:dvsd0} says that there is no recursive link between
$D_3(G)$ and $D_0(G)$. Can one show a superrecursive gap between
$D_a(G)$ and $D_b(G)$ for some $b>a>0$ or, at least, between $D(G)$
and $D_1(G)$?

Though the case of trees was thoroughly investigated
throughout the survey, this class of graphs deserves
further attention. One may expect that many logical
questions for trees are easier. Note in this respect that
the first-order theory of finite trees is decidable
due to Rabin \cite{Rab}. Nevertheless, we do not know, for example,
whether or not the logical depth $D(T)$ of a tree $T$ 
is a computable parameter
(while it is not hard to show that the logical width $W(T)$ 
is computable in logarithmic space).

Disappointingly, we were able to collect 
only a few results on the logical length for this survey.
From the fact that there are
$2^{(1/2+o(1))\, n^2}$ non-isomorphic graphs of order $n$, it is easy to derive that whp
$L(\rgraph)=\Omega\of{\frac{n^2}{\log n}}$. The obvious general upper bound is $O(n^2)$.
This leaves open the question what the logical length of a typical
graph is. Also, it would be very interesting to find explicit examples of graphs with
large $L(G)$. 
Pseudo-random graphs can be natural candidates.
For example, it is well known (Blass, Exoo, and Harary \cite{BEH})
that Paley graphs share the first-order properties of a truly random graph.

Furthermore, we can define the succinctness function with respect to the
logical length by $s(n)=\min\setdef{L(G)}{v(G)=n}$. Let $s_a(n)$ be the version
of $s(n)$ for the $a$-alternation logic. From Theorem \ref{thm:superrec}
and the relation \refeq{eq:LDprenex}, it follows that $s(n)$, and even $s_3(n)$,
can be incomprehensibly smaller than $n$: for any total recursive function $f$
we must have $f(s_3(n))<n$ infinitely often. On the other hand, the estimate
\refeq{eq:gnlow} in the proof of Theorem \ref{thm:tight} implies that
$s(n)=\Omega\of{\frac{\log n}{\log\log n}}$ for infinitely many $n$.
Moreover, the same argument shows that $s^*(n)=\Omega\of{\frac{\log n}{\log\log n}}$ 
for \emph{all} $n$, where $s^*(n)$ denotes the smoothed version of $s(n)$
similarly to \refeq{eq:smoothed}. How tight is the bound of $\Omega\of{\frac{\log n}{\log\log n}}$
in these statements? Another interesting problem is the behavior of the function $s_0(n)$
(recall that for $q_0(n)$ we know the exact asymptotics owing to Theorem \ref{thm:q0}).
Note in conclusion that techniques for estimating the length of a first-order formula
are worked out, e.g., by Adler and Immerman \cite{AIm}, Dawar et al.\ \cite{DGKS}, 
Grohe and Schweikardt~\cite{GSc}.

\subsection*{Acknowledgment}
We are grateful to Joel Spencer for
the fruitful collaboration on the subject of this survey and for
allowing us to use his unpublished ideas in the proof of 
Lemma \ref{lem:wsievergraph} and in Section~\ref{sss:padding}. 
We also thank Martin Grohe for his detailed comments, in particular, for
bringing up a succinctness result of Dawar et al.\ \cite{DGKS}
(Theorem \ref{thm:UnivTree} here) to our attention.

\renewcommand{\baselinestretch}{1}

\clearpage

\appendix

\section{Upper bound for the quantifier depth in the $k$-variable logic}

Theorem \ref{thm:Dk} can be somewhat improved. 

\begin{theorem}
  Let $k>2$ and suppose that graphs $G$ and $H$ are distinguishable
in the $k$-variable logic. Denote the number of vertices in $G$ by $n$ and assume that
$n\ge2$. Then
$D^k(G,H)\le n^{k-1}$.
Moreover, if $G$ is definable in the $k$-variable logic, then $D^k(G)\le n^{k-1}$.
\end{theorem}

\begin{proof}
As noted in the proof of Theorem \ref{thm:Dk},
\begin{equation}
  \label{eq:Dstab}
 D^{k}(G,H)\le\stabi{k-1}G+k,
\end{equation}
where $\stabi dG$ denotes the number of the first
iteration that does not refine the Weisfeiler-Lehman coloring of $V(G)^d$.
Since $d\ge2$, this coloring is from the very beginning not monochromatic.
More specifically, in the beginning $V(G)^d$ is partitioned into types of $d$-tuples
with respect to equality and adjacency relations. It follows that
\begin{equation}
  \label{eq:stab}
\stabi dG\le n^d-|C^0|, 
\end{equation}
where $|C^0|$ denotes the size of the initial
coloring of $V(G)^d$. If $n\ge d$, the number of equality types
is equal to the Bell number $B_d$. Counting the equality types alone
suffices if $k\ge4$; but in order to cover also $k=3$, we need a bit stronger bound.
Note that, unless $G$ is complete or empty, all but one equality
types are split further by taking adjacency into account.
Therefore, in this case we have $|C^0|\ge2B_d-1$.
Combining \refeq{eq:Dstab} and \refeq{eq:stab}, we conclude that
\begin{equation}
  \label{eq:Dk}
 D^{k}(G,H) \le n^{k-1}-2B_{k-1}+k+1,\text{ if }n\ge k-1. 
\end{equation}
Above we made the assumption that $G$ is neither complete nor empty.
Note that \refeq{eq:Dk} is actually true for all $G$.
For example, if $G$ is complete and $H$ is not, then $D^{k}(G,H) \le 2$.
If $G$ and $H$ are both complete or both empty and if they
are distinguishable with $k$ variables, then $D^{k}(G,H) \le k$.
This bound is within \refeq{eq:Dk}, as follows from the rough
estimate $B_d\le d^d/2$.

If $n\ge k-1$, the theorem immediately follows from \refeq{eq:Dk}.
If $n<k-1$, we just use the simple fact that any graph
with $n$ vertices is definable by the generic sentence \refeq{eq:def} of quantifier depth $n+1$.
Therefore, $D^k(G,H)\le n+1\le n^{k-1}$ also in this case (as $n\ge2$ and $k\ge3$).
\end{proof}

\end{document}